\documentclass[10pt, a4paper, leqno]{amsart}

\usepackage{geometry}
\usepackage{amsmath, amsthm, amsfonts, amssymb}
\allowdisplaybreaks[3]

\numberwithin{equation}{section}

\theoremstyle{plain}
\newtheorem{Theo}{Theorem}[section]
\newtheorem{Cor}[Theo]{Corollary}
\newtheorem{Lem}[Theo]{Lemma}
\newtheorem{Def}{Definition}[section]

\theoremstyle{remark}
\newtheorem{Rem}{Remark}[section]
\newtheorem*{Conj}{Conjecture}

\newcommand{\RR}{\mathbb{R}}
\newcommand{\NN}{\mathbb{N}}
\newcommand{\Ee}{\mathfrak{e}}
\newcommand{\Hh}{\mathcal{H}}
\newcommand{\Gg}{\mathcal{G}}
\newcommand{\Ll}{\mathrm{L}}
\newcommand{\C}{\mathrm{C}}
\newcommand{\R}{\mathrm{R}}
\newcommand{\eps}{\varepsilon}

\newcommand{\spp}{s^{\phi,\psi}}
\newcommand{\spu}{s^{\phi,u}}
\newcommand{\scvu}{s^{v,u}}
\newcommand{\sgn}{\operatorname{sgn}}
\newcommand{\card}{\operatorname{card}}

\usepackage{hyperref}

\title{Zeros of GKP sequences of polynomials}

\author{Antonio J. Dur\'an}
\address{Departamento de An\'a\-li\-sis Mate\-m\'a\-ti\-co and IMUS,
Uni\-ver\-sidad de Se\-vi\-lla,
Se\-vi\-lla, Spain} 
\email{duran@us.es}
\author{Mario P\'erez}
\address{Departamento de Mate\-m\'a\-ti\-cas and IUMA,
Universidad de Zara\-goza,
Zaragoza, Spain} 
\email{mperez@unizar.es}
\author{Juan L. Varona}
\address{Departamento de Mate\-m\'a\-ti\-cas y Com\-pu\-ta\-ci\'on,
Universidad de La Rioja,
Lo\-gro\-\~no, Spain} 
\email{jvarona@unirioja.es}

\dedicatory{To the memory of our friend Jes\'us Guillera}

\begin{document}

\begin{abstract}
Given two sequences $\phi=(\phi_i)_{i\ge 1}$ and $\psi=(\psi_i)_{i\ge 1}$ and numbers $a,b,c$, we introduce the GKP sequence of polynomials $(p_n)_n$ using the following recurrence formula:
$p_0 = 1$ and for $n\ge 1$
\[
  p_{n}(x) = (ax^2+bx+c) p_{n-1}'(x) + (\phi_{n} + \psi_{n} x)p_{n-1}(x),
\]
where we assume that $ax^2+bx+c$ has two different real zeros.

Tangent, Secant, Eulerian or Jacobi polynomials are examples of GKP sequences of polynomials. In this paper, under mild assumptions we prove that the zeros of the polynomials $p_n$ are real, simple and live between the zeros of $ax^2+bx+c$. Moreover, the zeros of $p_{n+1}$ interlace the zeros of $p_n$. We study in detail the cases when $\psi$ is constant, and $\phi=(\phi_i)_{i\ge 1}$ is constant for $i$ big enough, proving, among other results, asymptotics for the leftmost and rightmost zeros of~$p_n$.
\end{abstract}

\maketitle

\section{Introduction and results}

The purpose of this paper is to study properties of zeros of GKP sequences of polynomials.

\begin{Def}
Given real numbers $a,b,c\in \RR$ such that the polynomial $ax^2+bx+c$ has two different real zeros, and two sequences of real numbers
$\psi= (\psi_i)_{i \ge 1}$, $\phi = (\phi_i)_{i \ge 1}$,
with $\psi_{i} \ne -a(i-1)$, $i\ge 1$, the GKP sequence of polynomials
$p_n(x)$ is defined by $p_0 = 1$ and
\begin{equation}
\label{sn+1i}
  p_{n}(x) = (ax^2+bx+c) p_{n-1}'(x) + (\phi_{n} + \psi_{n} x)p_{n-1}(x), \quad n \ge 1.
\end{equation}
\end{Def}

It is clear that $p_n$ is a polynomial of degree $n$ with leading coefficient $\prod_{j=1}^n (a(j-1)+\psi_{j})$.

A simple computation shows that if we write $p_n(x) = \sum_{j=0}^n a_{n,j}x^j$ then the two-index sequence $a_{n,j}$ satisfies the recurrence relation given by $a_{0,0}=1$ and, for $n,j\ge 1$,
\begin{equation}
\label{gkp1}
  a_{n,j} = c(j+1)a_{n-1,j+1}+(bj+\phi_n)a_{n-1,j}+(aj+\psi_n-a)a_{n-1,j-1}.
\end{equation}

The initials \textbf{GKP} in the expression ``GKP sequence of polynomials'' are for R.~L. \textbf{G}raham, D.~E. \textbf{K}nuth, and O. \textbf{P}atashnik. They proposed the following problem in \textit{Concrete Mathematics}
\cite[Problem 6.94, p.~319]{Knuth-libro}: Develop a general theory of the solutions to the two-parameter recurrence $a_{0,0}=1$,
\begin{equation}
\label{gkp2}
  a_{n,j} = (An+Bj+C)a_{n-1,j}+(A'n+B'j+C')a_{n-1,j-1},\quad n,j\ge 0
\end{equation}
(assuming that $a_{n,j}=0$ when $n<0$ or $j<0$).

Many relevant sequences of numbers satisfy recurrences of the form~\eqref{gkp2}: binomial coefficients, both kinds of Stirling numbers, Eulerian numbers, Lah numbers, among others. Since Graham, Knuth and Patashnik proposed this problem, some results on it have appeared using different approaches: \cite{BSV, Neu, Spi, Th1, Th2, Th3, Wilf2}.

Comparing~\eqref{gkp1} and~\eqref{gkp2}, we see that when $B' \ne 0$, any solution of the GKP problem are the coefficients $a_{n,j}$ of the GKP sequence of polynomials given by
\[
  c=0,\quad
  \begin{cases}
  \phi_n=An+C,& \\
  b=B,&
  \end{cases}
  \quad
  \begin{cases}
  \psi_n=A'n+a+C',& \\
  a=B'.&
  \end{cases}
\]
The case $B'=0$ and $B \ne 0$ was considered in \cite{Duran-Bell,dzol}, and the case $B'=B=0$ and $c \ne 0$ in~\cite{dzoh}.

Actually, we can consider in~\eqref{sn+1i} a particular polynomial with two different real zeros (for instance, $1-x^2$) instead of the arbitrary polynomial $ax^2+bx+c$, because the corresponding GKP sequences of polynomials are connected through a normalization plus a linear change of variable (see Remark~\ref{misi}).

\begin{Def}
Given two sequences of real numbers
$\psi= (\psi_n)_{n \ge 1}$, $\phi = (\phi_n)_{n \ge 1}$,
with $\psi_{n} \ne n-1$, $n\ge 1$, we define the sequence of polynomials
$\spp_n(x)$ as $\spp_0(x)=1$ and
\begin{equation}
\label{sn+1}
  \spp_{n+1}(x) = (1-x^2) (\spp_n)'(x) + (\phi_{n+1} + \psi_{n+1} x)\spp_n(x), \quad n \ge 1.
\end{equation}
\end{Def}

In order to simplify the notation, we will sometimes write $s_n$ instead of~$\spp_n$.

It is clear that $\spp_n$ is a polynomial of degree $n$ with leading coefficient $\prod_{j=1}^n(-j+1+\psi_{j})$.

For a real number $v\in \RR$, the constant sequence $\phi_i=v$, $i\ge 1$, will be denoted by $v$, and the same with $\psi_i=u$, $i\ge 1$. In particular, we will very often use the sequences $\scvu_{n}(x)$.

The content of the paper is as follows.

In Section~\ref{sec2}, we show that some significant examples of polynomials, such as Tangent and Secant, Eulerian and Jacobi polynomials, are particular cases of GKP polynomials.

In Section~\ref{sec3}, we prove that under mild assumptions, the GKP polynomials have nice zero properties, as the following theorem shows.

\begin{Theo}
\label{ptci}
Assuming that $|\phi_n| + \psi_n < 0$ for every $n \ge 0$, then all the zeros of the polynomial $\spp_n$ \eqref{sn+1} are simple and belong to the interval $(-1,1)$. Moreover, the zeros of $s_{n+1}^{\phi,\psi}$ interlace the zeros of~$\spp_n$ (for the definition of the interlacing property see Definition~\ref{dip} below).
\end{Theo}

The case when $\psi$ is a constant sequence, $\psi_i=v\in \RR$, $i\ge 1$, is specially interesting because then the coefficients of $\spp_n$ are symmetric functions of the parameters $\phi_1,\dots,\phi_n$. As a consequence, we will show in Section~\ref{sec4} that the zeros of $\spp_n$ have some extra properties, such as the monotonicity with respect to the parameters $\phi$, and some others.

In Section~\ref{sec5} we study the case when $\psi$ and $\phi$ are constant sequences, $\psi_i=v\in \RR$, $\phi_i=u\in \RR$, $i\ge 1$. We then prove the following asymptotic behaviour for the zeros of the polynomials $\scvu_n(2x+1)$.

\begin{Theo}
\label{thchi} 
Assume $|v|+u<0$ and let $\zeta_{n,j}$, $j=1,\dots, n$, the zeros of $s_n^{v,u}(2x+1)$ arranged in decreasing size. We then have for the rightmost zeros,
\begin{equation}
\label{rmzi}
  \lim_{n\to \infty} \frac{\zeta_{n,j}}{\frac{j}{u-j+1}\left(\frac{u+v-2j+2}{u+v-2j}\right)^n} = 1,
  \quad j\ge 1.
\end{equation}
And for the leftmost zeros,
\begin{equation}
\label{lmzi}
  \lim_{n\to \infty} \frac{\zeta_{n,n-j+1}}{-1-\frac{j}{u-j+1}
  \left(\frac{u-v-2j+2}{u-v-2j}\right)^n} = 1,
  \quad j\ge 1.
\end{equation}
\end{Theo}

This is a consequence of a general and powerful result on zeros of polynomials $P_n(x) = \sum_{j=0}^na_{n,j}x^j$ whose coefficients $a_{n,j}$ behave as $F(j)G^n(j)$, where $F$ and $G$ satisfy
$G(0),G(1)>0$ and
\[
  |F(j-1)F(j+1)|\le F^2(j),\quad 0 < G(j-1)G(j+1) < G^2(j)
\]
(see Theorem~\ref{sip}). This result, which is interesting in its own right, will be proved in the Appendix.

Using the same tool, we study, in Section~\ref{sec6}, the case when $\psi$ is constant and the sequence $\phi=(\phi_i)_i$ is constant for $i$ big enough. This is equivalent to studying finite linear combinations of the polynomials $\scvu_n$ of the following type: we associate to the real numbers $u$, $v$, $\gamma_j$, $0\le j\le K$, with $|v|+u<0$, the polynomials
\begin{equation}
\label{polqni}
  q_n^{v,u}(x) = \sum_{j=0}^K\gamma_j\scvu_{n-j}(2x+1),\quad n\ge K.
\end{equation}
We also consider the polynomial
\begin{equation}
\label{polpi}
  P(x) = \sum_{j=0}^K \gamma_jx^{K-j}.
\end{equation}

The main result of this section is the following theorem. Assume that
\begin{equation}
\label{conci}
  P(u+v- 2l),P(-u-v+2l) \ne 0,\quad l\in \NN.
\end{equation}
Write $N^+$, $N^-$ for the cardinal of the sets
\begin{gather*}
  \{l: \text{$l\ge 0$, $P(u+v-2l)P(u+v-2l-2)<0$}\}, \\
  \{l: \text{$l\ge 0$, $P(-u-v+2l)P(-u-v+2l+2)<0$}\},
\end{gather*}
respectively.

\begin{Theo}
\label{th6.3i}
Assume $|v|+u<0$ and that the polynomial~$P$ \eqref{polpi} has only real zeros and satisfies~\eqref{conci}. 
\begin{enumerate}
\item\label{th6.3i-1}
If the zeros of $P$ \eqref{polpi} live in the interval $(v+u,v-u)$, then all the zeros of the polynomial $q_n^{v,u}$ \eqref{polqni} are real and simple for $n\ge K$, and they live in the interval $(-1,0)$. Moreover, the zeros of the polynomial $q_{n+1}^{v,u}$ interlace the zeros of~$q_n^{v,u}$.
\item\label{th6.3i-2}
Otherwise, for $n$ big enough, the polynomial $q_n^{v,u}$ \eqref{polqni} has only real zeros. Moreover, $n-N^+-N^-$ of these zeros live in $(-1,0)$, $N^+$ in $(0,+\infty)$ and the other $N^-$ in $(-\infty,-1)$, and the zeros of $q_{n+1}^{v,u}$ in $(-1,0)$, and $(0,+\infty)$ interlace the zeros of $q_n^{v,u}$ in $(-1,0)$ and $(0,+\infty)$, respectively, and the zeros of $q_{n}^{v,u}$ in $(-\infty,-1)$ interlace the zeros of $q_{n+1}^{v,u}$ in $(-\infty,-1)$.
\end{enumerate}
\end{Theo}

We also provide the asymptotic for the rightmost and leftmost zeros of $q_n^{v,u}$ (see Theorem~\ref{onl3}).

\section{Examples of GKP sequences of polynomials}
\label{sec2}

We start noticing that the polynomial $p_n$ \eqref{sn+1i} depends only on $\phi_i,\psi_i$, $1\le i\le n$ (and on $a,b,c$). Hence, one can associate the finite family of polynomials $p_j$, $0\le j\le n$, to the finite sequences $(\psi_i)_{i=1}^n$ and $(\phi_i)_{i=1}^n$.

\begin{Rem}
\label{misi}
An easy computation shows that any GKP sequence of polynomials~\eqref{sn+1i} can be written in terms of the sequence $s_n^{\tilde \phi,\tilde \psi}$ \eqref{sn+1} as follows:
\[
p_n(x) = \gamma^n s_n^{\tilde \phi,\tilde \psi}(\mu x + \nu),
\]
where $\omega_i$, $i=1,2$, are two different real zeros of $ax^2+bx+c$ with $\omega_2>\omega_1$ and
\begin{gather*}
  \gamma = -\frac{a}{2}(\omega_2-\omega_1),
  \qquad \mu = \frac{2}{\omega_2-\omega_1},
  \qquad \nu = -\frac{\omega_2+\omega_1}{\omega_2-\omega_1},
  \\
  \tilde{\phi}_n = \frac{-1}{a(\omega_2-\omega_1)} \Big( 2\phi_n + (\omega_2+\omega_1)\psi_n \Big),
  \qquad \tilde{\psi}_n = \frac{-1}{a} \psi_n.
\end{gather*}
\end{Rem}

It is easy to check that
\begin{align}
\label{va-1}
  \spp_n(-1) = \prod_{j=1}^n(\phi_{j}-\psi_j),
  \\
\label{va1}
  \spp_n(1) = \prod_{j=1}^n(\phi_{j}+\psi_j).
\end{align}
The factor $1-x^2$ in~\eqref{sn+1} implies that
if $\spp_2(\pm 1) = 0$ then
$\spp_2(x) = \psi_1 (1 - \psi_2)(1-x^2)$, and
since $\psi_1 (1 - \psi_2) \ne 0$, the recursive definition of $\spp_n$ proves that $\spp_n(\pm 1) = 0$ for each $n \ge 2$. We can try to avoid these common zeros at $\pm 1$ as follows.
We write
\begin{equation}
\label{zemm}
  \spp_{n+2}(x) = \psi_1 (1 - \psi_2)(1-x^2) t_n(x),
\end{equation}
where the polynomials $t_n(x)$, $n \ge 0$, with $t_0 = 1$, are defined by
\[
  t_{n+1}(x) = (1-x^2) t_n'(x)
  + \Big(\phi_{n+3} + (\psi_{n+3}-2) x\Big) t_n(x),\quad n\ge 1.
\]
In other words, $t_n(x) = s_{n}^{\tilde \phi,\tilde\psi}(x)$ with the new parameter sequences
\begin{align*}
  \tilde\phi &= (\phi_3, \phi_4, \phi_5, \dots),
  \\
  \tilde\psi &= (\psi_3-2, \psi_4-2, \psi_5-2, \dots).
\end{align*}

Many sequences of polynomials can be represented using GKP sequences of polynomials, as the following three illustrative examples show.

\subsection{Tangent and Secant polynomials}

Tangent polynomials are defined by $P_0=1$, and
\[
  \frac{d^n}{dx^n}\tan (x) = P_{n+1}(\tan x), \quad n\ge 0.
\]
They were considered by Donald E. Knuth and Thomas J. Buckholtz \cite{Knuth-B} (see also
\cite[(6.94), p.~287]{Knuth-libro}).
For the purpose of this paper, it is better to consider the sequence $T_n(x) = -i^nP_n(ix)$ (corresponding to the hyperbolic tangent). It is clear that they can be defined recursively by $T_0=1$, $T_1(x)=x$ and
\[
  T_{n+1}(x) = (1-x^2)T_n'(x),\quad n\ge 1.
\]
So, we have
\[
  T_n(x) = s_n^{0,\hat\psi}(x),
\]
where, as explained above, $0$ is denoting the sequence $\phi_i=0$, $i\ge 1$, and
\[
  \hat\psi_i = 
  \begin{cases} 1,&i=1,\\ 0,&i\ge 2.\end{cases}
\]
Since $T_2(x) = 1-x^2$, proceeding as above (see~\eqref{zemm}), we have
\[
  T_n(x) = (1-x^2)s_{n-2}^{0,-2}(x), \quad n \ge 2.
\]

Secant polynomials are defined by $Q_0=1$ and
\[
\frac{d^n}{dx^n}\sec (x) = Q_{n}(\tan x)\sec x,\quad n\ge 0.
\]
They were considered by D. E. Knuth and T. J. Buckholtz in \cite{Knuth-B}. 
Again, for the purpose of this paper it is better to consider the sequence $R_n(x) = i^nQ_n(ix)$. 
It is clear that they can be defined recursively by $R_0=1$ and
\[
  R_{n+1}(x) = (1-x^2)R_n'(x)-xR_n(x),\quad n\ge 1.
\]
So, we have
\[
  R_n(x) = s_n^{0,-1}(x),
\]
where $0$ and $-1$ denote the sequences $\phi_i=0$, $i\ge 1$, and
$\psi_i=-1$, $i\ge 1$.

Either in the form of tangent and secant or hyperbolic tangent and hyperbolic secant polynomials, these families of polynomials have been considered in different contexts, ranging from combinatorics to applications to the Korteweg-de Vries equation (see, for instance, \cite{Knuth-libro, Knuth-B, Hoff1995, Hoff1999, GV, Boy}), although we have not found a study of their zeros.

\subsection{Eulerian polynomials}

The Eulerian polynomials $(\Ee_n)_n$ are defined by $\Ee_0=1$ and
\begin{equation*}
  \Ee_n(x) = \sum_{j=0}^{n-1} \left\langle \begin{matrix} n \\ j \end{matrix} \right\rangle x^j,
  \quad n\ge 1,
\end{equation*}
where $\left\langle \begin{smallmatrix} n \\ j \end{smallmatrix} \right\rangle$, $0\le j\le n$, denote the Eulerian numbers
\[
  \left\langle \begin{matrix} n \\ j \end{matrix} \right\rangle
  = \sum_{i=0}^j(-1)^{i}\binom{n+1}{i}(j+1-i)^n,
  \quad 0\le j\le n.
\]
Eulerian numbers were introduced by L. Euler \cite{Eul} and have interest in combinatorics and some other areas.
At $x=-1$, the Eulerian polynomials provide the value of the Riemann zeta function at the negative integers
\[
  \Ee_n(-1) = (2^{n+1}-4^{n+1})\zeta(-n),\quad n\ge 0.
\]
Since the Eulerian polynomials satisfy
\[
  \Ee_n(x) = x(1-x)\Ee_{n-1}'(x)+(1+(n-1)x)\Ee_{n-1}(x), \quad n\ge 1,
\]
they are the GKP polynomials for $a=-1$, $b=1$, $c=0$, $\phi_j=1$, $\psi_j=j-1$.

However, the Eulerian polynomials are also related to the polynomials $s_{n}^{0,-2}$, $n\ge 0$. Indeed, consider the modified Eulerian polynomials defined by
\begin{equation*}
  e_n(x) = x^n\Ee_n(1+1/x).
\end{equation*}
Using
\[
  j!\, S(n,j) = \sum_{k=0}^n \left\langle \begin{matrix} n \\ k \end{matrix} \right\rangle \binom{k}{n-j}
\]
(see \cite[(6.39), p.~269]{Knuth-libro}), we have
\begin{equation*}
  e_n(x) = \sum_{j=0}^n j!\, S(n,j)x^j,
\end{equation*}
where $S(n,j)$, $0\le j\le n$, are the Stirling numbers of the second kind. The polynomials $e_n$, $n\ge 0$, are also called geometric or Fubini polynomials.

A simple computation using \cite[(7a), p.~10]{Hir} shows that the polynomials $(e_n)_n$ satisfy
the recurrence
\begin{equation*}
  e_{n+1}(x) = x\left(1+(1+x)\frac{d}{dx}\right)e_n(x),\quad n\ge 0, \quad e_0=1.
\end{equation*}
From where it is easy to conclude that
\begin{equation*}
  e_{n+1}(x) = \frac{x}{(-2)^{n}}s_{n}^{0,-2}(1+2x)
\end{equation*}
(see also \cite{Boy, dzb}).

\subsection{Jacobi polynomials}

For the Jacobi polynomials $P_n^{(\alpha,\beta)}(x)$, we have
\[
  (-1)^n 2^n n!\,P_n^{(\alpha,\beta)}(x) = s_n^{\phi,\psi^n}(x),
\]
where
\[
  \phi_i = \beta-\alpha, \quad i\ge 1, \quad \psi_i^n = -(\alpha+\beta+2(n-i+1)),\quad i=1,\dots, n.
\]
This is a consequence of the backward shift operator for the Jacobi polynomials \cite[(9.8.8), p.~218]{KLS}.

Notice that for a fixed $n$, we can consider the polynomials $s_m^{\phi,\psi^n}(x)$, $m\le n$. 
It is easy to see that
$s_m^{\phi,\psi^n}(x) = (-1)^m 2^m m!\,P_m^{(\alpha+n-m,\beta+n-m)}(x)$.

\section{Zeros of the GPK polynomials}
\label{sec3}

Throughout this paper, the interlacing property is defined as follows.

\begin{Def}
\label{dip}
Given two finite sets $U$ and $V$ of real numbers ordered by size, we say that $U$ interlaces $V$ if $\min(U)<\min(V)$ and between any two consecutive elements of any of the two sets there exists one element of the other (in particular, this gives $U\cap V=\emptyset$).
\end{Def}

Note that if $U$ interlaces $V$, then either $\card(U) = \card(V)$, and then $\max(U) < \max(V)$, or $\card(U) = 1 + \card(V)$, and then $\max(U) > \max(V)$. Observe also that the interlacing property is not symmetric, due to the condition $\min(U) < \min(V)$.

We next prove that under mild conditions on the parameters, the zeros of the GKP polynomials enjoy nice properties.

Given a polynomial $p$ (non-constant, with real coefficients), let us define
\[
  q(x) = (1-x^2) p'(x) + (a+bx) p(x),
\]
for some $a, b \in \RR$. It is easy to see that $q$ can not be the null polynomial, except when $p(x) = (1-x)^n (1+x)^m$, $a = m-n$, $b = n+m$. The additional assumption that $b \ne \deg p$ would assure that $\deg q = 1 + \deg p$. Each zero of $p$ with multiplicity $k$ is a zero of $q$ with multiplicity $k-1$ (at least $k-1$ if the zero is $\pm 1$), so $q$ must have some additional zeros. The next two lemmas give some information about them.

\begin{Lem}
\label{depaq1}
Let $p$ be a polynomial (non-constant, with real coefficients), $a, b \in \RR$, and define
\[
  q(x) = (1-x^2) p'(x) + (a+bx) p(x).
\]
Then, between any two consecutive zeros of $p$ in $(-1,1)$ there exists at least one zero of~$q$.
\end{Lem}

\begin{proof}
Let $c < d$ be two consecutive zeros of $p$, not necessarily simple. Then, $p(x) = (x-c)^n (d-x)^m r(x)$, where $n, m \in \NN$ and $r(x)$ is some polynomial not vanishing in $[c,d]$ and, therefore, with constant sign there. Thus, $q(x) = (x-c)^{n-1} (d-x)^{m-1} s(x)$, where
\begin{multline*}
  s(x) = (1-x^2) \Big( n(d-x) r(x) - m(x-c)r(x) + (x-c)(d-x) r'(x) \Big)
  \\
  + (a+bx)(x-c)(d-x)r(x).
\end{multline*}
In particular,
\begin{align*}
  s(c) &= (1-c^2) n (d-c) r(c),
  \\
  s(d) &= - (1-d^2) m (d-c) r(d).
\end{align*}
Since $r$ has constant sign in $[c,d]$ and $c, d \in (-1,1)$, we deduce that $s(c) s(d) < 0$, so that $s(x)$ has some zero in $(c,d)$, which is also a zero of $q(x)$.
\end{proof}

Under the condition $b + |a| < 0$, the above result can be improved so as to include all possible relative locations of $c$, $d$, $-1$ and~$1$. In some cases (for instance, if $c < -1$, $1 < d$, which is considered in the above result), more than one zero appears. To sum up, if we cut the interval $(c,d)$ at the points $\pm 1$, each resulting interval contains some zero of~$q$.

\begin{Lem}
\label{depaq2}
Let $p$ be a polynomial (non-constant, with real coefficients), $a, b \in \RR$ such that $b + |a| < 0$, and define
\[
  q(x) = (1-x^2) p'(x) + (a+bx) p(x).
\]
Let $Z(p)$ the set of real zeros of~$p$. If $c < d$ are two consecutive elements of $Z(p) \cup \{\pm 1\}$, then $q$ has an odd number of zeros in $(c,d)$.
\end{Lem}

\begin{proof}
Observe that $(c,d)$ does not contain any element of $Z(p) \cup \{\pm 1\}$, since $c$ and $d$ are consecutive elements.
We proceed in four steps depending on whether any of $c$, $d$ are $\pm 1$. The proof is slightly different, but the arguments are the same.

\medskip
\noindent
\textit{Step 1.} 
Assume that neither $c$ nor $d$ are~$\pm 1$. Necessarily, $c$ and $d$ are zeros of $p$. Then, $p(x) = (x-c)^n (d-x)^m r(x)$, where $n, m \in \NN$ and $r(x)$ is some polynomial not vanishing in $[c,d]$ and, therefore, with constant sign there. Thus, $q(x) = (x-c)^{n-1} (d-x)^{m-1} s(x)$, where
\begin{multline*}
  s(x) = (1-x^2) \Big( n(d-x) r(x) - m(x-c)r(x) + (x-c)(d-x) r'(x) \Big)
  \\
  + (a+bx)(x-c)(d-x)r(x).
\end{multline*}
In particular,
\begin{align*}
  s(c) &= (1-c^2) n (d-c) r(c),
  \\
  s(d) &= - (1-d^2) m (d-c) r(d).
\end{align*}
Since $[c, d ] \cap \{\pm 1\} = \emptyset$, either $c, d \in (-1,1)$ or $c, d < -1$, or $c, d > 1$. In any case $\sgn(1-c^2) = \sgn(1-d^2)$. Since $r$ has a constant sign in $[c,d]$, we deduce that $s(c) s(d) < 0$, so that $s(x)$ has an odd number of zeros in $(c,d)$, which is also a zero of $q(x)$. The condition $b + |a| < 0$ is not necessary in this case.

\medskip
\noindent
\textit{Step 2.} 
Assume that $d = -1$. Then, $p(x) = (x-c)^n (d-x)^m r(x)$ as before, with the only difference that $m \in \NN$ or $m = 0$ depending on whether $p(d) = 0$ or not. In any case, we can now write $q(x) = (x-c)^{n-1} (d-x)^m s(x)$, where
\[
  s(x) = (1-x^2) n r(x) + (1-x)m(x-c)r(x) + (1-x^2)(x-c) r'(x)
  + (a+bx)(x-c) r(x)
\]
and $r(x)$ has constant sign of $[c,d]$. In particular,
\begin{align*}
  s(c) &= (1-c^2) n r(c),
  \\
  s(d) &= - (b-a-2m) (d-c) r(d).
\end{align*}
Therefore
\[
  s(c) s(d) = - (b-a-2m) (d-c) (1-c^2) n r(c) r(d) < 0,
\]
due to the condition $b - a < 0$. Again, this gives an odd number of zeros of $s$ and $q$ in $(c,d)$.

\medskip
\noindent
\textit{Step 3.} 
Assume that $c = -1 < d < 1$. Let us write $p(x) = (x-c)^n (d-x)^m r(x)$ as before, where $n \in \NN$ or $n = 0$ depending on whether $p(c) = 0$ or not, and $m \in \NN$. As always, $r(x)$ has constant sign in $[c,d]$. Then, $q(x) = (x-c)^n (d-x)^{m-1} s(x)$, with
\[
  s(x) = (1-x) n(d-x) r(x) - (1-x^2) m r(x) + (1-x^2) (d-x) r'(x)
  + (a+bx)(d-x)r(x).
\]
Now
\begin{align*}
  s(c) &= - (b-a-2n) (d-c) r(c),
  \\
  s(d) &= - (1-d^2) m r(d),
\end{align*}
and
\[
  s(c) s(d) = (b-a-2n) (d-c) (1-d^2) m r(c) r(d) < 0,
\]
due to the condition $b - a < 0$. This proves that $s$ and $q$ have an odd number of zeros in $(c,d)$.

\medskip
\noindent
\textit{Step 4.} 
Assume now that $c = -1$ and $d = 1$. Then $p(x) = (x-c)^n (d-x)^m r(x)$ as before, where now $n, m \in \NN \cup \{ 0 \}$ and $r(x)$ has constant sign in $[-1,1]$. Then, $q(x) = (x-c)^n (d-x)^m s(x)$, where
\[
  s(x) = n(d-x) r(x) - m(x-c)r(x) + (x-c)(d-x) r'(x)
  + (a+bx) r(x).
\]
Now
\begin{align*}
  s(c) &= - (b-a-2n) r(c),
  \\
  s(d) &= (b + a - 2m) r(d),
\end{align*}
and
\[
  s(c) s(d) = - (b-a-2n) (b + a - 2m) r(c) r(d) < 0,
\]
where both conditions $b \pm a < 0$ are used. This proves that $s$ and $q$ have an odd number of zeros in $(c,d)$.

\medskip
\noindent
\textit{Step 5.} 
Finally, if $-1 < c < 1 = d$ or $c = 1$, the proofs are similar to the cases $c = -1 < d < 1$ and $d = -1$, respectively. The condition $a + b < 0$ is used here.
\qedhere
\end{proof}

\begin{Lem}
Let $p$ be a polynomial (non-constant, with real coefficients), $a, b \in \RR$ such that $b + |a| < 0$, and define
\[
  q(x) = (1-x^2) p'(x) + (a+bx) p(x).
\]
Assume that all the zeros of $p$ are real. Then, all the zeros of $q$ are real and can be described as follows:
\begin{itemize}
\item If $c \ne \pm 1$ is a zero of $p$ with multiplicity $n$, then it is a zero of $q$ with multiplicity $n-1$.

\item If $-1$ is a zero of $p$, then it is a zero of $q$ with the same multiplicity. If $1$ is a zero of $p$, then it is a zero of $q$ with the same multiplicity.

\item Let $Z(p)$ be the set of zeros of $p$; if $c < d$ are consecutive elements of the set $Z(p) \cup \{\pm 1\}$, then $q$ has exactly one zero in $(c,d)$, with multiplicity $1$ (that is, one simple zero).
\end{itemize}
\end{Lem}

\begin{proof}
Let us note that under the condition $b + |a| < 0$, $q$ is a polynomial with degree $1 + \deg p$.

Let $Z(p) \setminus \{\pm 1\} = \{\xi_1 < \xi_2 < \dots < \xi_j\}$ be the set of zeros of $p$ not equal to $\pm 1$ (it could be the empty set), and let $n_1, \dots, n_j$ be their respective multiplicities. Let $n_{-}$ and $n_{+}$ be the multiplicities of $-1$ and $1$, respectively, as zeros of $p$, where $n_{-} = 0$ means that $-1$ is not a zero of $p$, and the same with~$n_{+}$. Then,
\[
  n_1 + \dots + n_j + n_{-} + n_{+} = \deg p.
\]
It is immediate that each $\xi_k$ is a zero of $q$ with multiplicity \emph{at least} $n_k-1$. And if $\pm 1$ is a zero of $p$, then it is also a zero of $q$ with \emph{at least} the same multiplicity. The previous lemma proves that if $c < d$ are consecutive elements of the set $Z(p) \cup \{\pm 1\}$, then $q$ has \emph{at least} one zero in $(c,d)$, which gives at least $j+1$ zeros. So, we only need to remove those ``at least''. Thus, we already have
\[
  (n_1 - 1) + \dots + (n_j - 1) + n_{-} + n_{+} + j + 1
  = n_1 + \dots + n_j + n_{-} + n_{+} + 1
  = \deg q
\]
zeros, taking into account their multiplicities, and no more zeros or multiplicities can occur.
\end{proof}

As a particular case of the previous lemma, we have the following.

\begin{Lem}
\label{depaq}
Let $p$ be a polynomial with real coefficients and define
\[
  q(x) = (1-x^2) p'(x) + (a + b x) p(x),
\]
where $a$ and $b$ are real numbers such that $|a| + b < 0$. If all the zeros of $p$ are simple and belong to the interval $(-1,1)$, then all the zeros of $q$ are simple, belong to the interval $(-1,1)$, and interlace the zeros of~$p$.
\end{Lem}

We are now ready to prove Theorem~\ref{ptci} in the introduction, which shows that under the assumptions $|\phi_n| + \psi_n < 0$, $n\ge 0$, the zeros of $\spp_n$ have nice properties.

\begin{proof}[Proof of Theorem~\ref{ptci}]
Since $s_1(x) = \phi_1 + \psi_1 x$, the condition $\psi_1 + |\phi_1| < 0$ proves that its only zero is $- \frac{\phi_1}{\psi_1} \in (-1,1)$. Now the theorem follows easily from Lemma~\ref{depaq} by induction on~$n$.
\end{proof}

\section{On the symmetric dependence of \texorpdfstring{$\spp_n$}{snphipsi} on the parameters}
\label{sec4}

The case when the sequence $\psi$ is constant is especially interesting, as we will see later. It turns out that in this case the coefficients of $\spp_n$ are symmetric functions of $\phi_1, \phi_2,\allowbreak \dots, \phi_n$. A natural question is whether this symmetry in $\phi_1, \phi_2,\allowbreak \dots, \phi_n$ also holds for other sequences~$\psi$. In other terms, for which sequences $\psi$ is $\spp_n$ invariant under any permutation of $(\phi_1, \phi_2, \dots, \phi_n)$; in turn, this reduces to the invariance under any iteration of single permutations $(\phi_j, \phi_{j+1}) \mapsto (\phi_{j+1}, \phi_j)$.

So for each $a, b \in \RR$, let us take
\[
  \Omega(a,b) p(x) = (1-x^2) p'(x) + (a + bx) p(x).
\]
It is easy to check that, for any $a, b, c, d \in \RR$,
\begin{equation}
\label{Omegacdab}
\begin{aligned}
  &\Omega(c,d) \Omega(a,b) p(x)
  = - 2x (1-x^2) p'(x) + (1-x^2) p''(x) + b(1-x^2) p(x)
  \\
  &\qquad\qquad + (a + c + bx + dx) (1-x^2) p'(x)
  + (c+dx)(a+bx) p(x). 
\end{aligned}
\end{equation}
Thus, given a not null polynomial $p(x)$, it is plain that
\[
  \Omega(c,d) \Omega(a,b) p = \Omega(a,d) \Omega(c,b) p
  \iff (a-c) (b-d) = 0.
\]
Since $\spp_{j+1} = \Omega(\phi_{j+1},\psi_{j+1}) \Omega(\phi_j, \psi_j) \spp_{j-1}$,
it follows that $\spp$ is invariant under a single permutation $(\phi_j, \phi_{j+1}) \mapsto (\phi_{j+1}, \phi_j)$ if and only if
\[
  (\phi_j - \phi_{j+1}) (\psi_j - \psi_{j+1}) = 0.
\]
And, excluding the trivial case when the sequence $\phi = (\phi_1, \phi_2, \dots)$ is constant, we deduce that $\spp$ is invariant under any iteration of those single permutations if and only if $\psi = (\psi_1, \psi_2, \dots)$ is constant.

We might ask also for which sequences $\phi$ is $\spp_n$ a symmetric function of $\psi_1,\allowbreak \psi_2,\allowbreak \dots,\allowbreak \psi_n$, that is, invariant under any iteration of single permutations $(\psi_j, \psi_{j+1}) \mapsto (\psi_{j+1}, \psi_j)$. We conclude with no effort from~\eqref{Omegacdab} that, given a not null polynomial~$p$,
\[
  \Omega(c,d) \Omega(a,b) p = \Omega(c,b) \Omega(a,d) p
  \iff b = d.
\]
Therefore, $\spp_n$ is a symmetric function of $\psi_1,\psi_2, \dots, \psi_n$ only in the trivial case when $\psi$ is a constant sequence. The same conclusion follows if we look for joint symmetry in both parameters, that is, invariance under iteration of single permutations $\big((\phi_j,\psi_j), (\phi_{j+1},\psi_{j+1}) \big) \mapsto \big((\phi_{j+1},\psi_{j+1}), (\phi_j,\psi_j) \big)$.

As we have just proved, when the sequence $\psi_n$ is constant, the polynomial $\spp_n$ is a symmetric function of $\phi_1,\dots, \phi_n$ (and reciprocally). This case is of particular importance because then the zeros of the polynomials $\spp_n$ enjoy more properties. Hence, let us consider the case when $\psi_n = u$ for every $n \ge 1$, for some fixed $u \in \RR$. If we define the first order differential operator $\Lambda_u$ by
\[
  \Lambda_up(x) = (1-x^2)p'(x)+uxp(x),
\]
then the polynomials $s_n^{\phi,u}$ are given by the recursion
\begin{align*}
  s_0^{\phi,u} &= 1,
  \\
  s_{n+1}^{\phi,u} &= (\Lambda_u + \phi_{n+1} I) s_n^{\phi,u},
\end{align*}
where $I$ is the identity operator $I p = p$. Thus,
\begin{equation}
\label{snLphi}
  s_n^{\phi,u}
  = (\Lambda_u + \phi_n I) (\Lambda_u + \phi_{n-1} I) \cdots (\Lambda_u + \phi_1 I) s_0.
\end{equation}
Since $\Lambda_u$ commutes with all the operators $\phi_j I$, this proves that
\[
  s_n^{\phi,u} = \sum_{j=0}^n \Phi_{n-j}^n \Lambda_u^j s_0,
\]
where $\Phi_{n-j}^n$ are the coefficients given by
\[
  \prod_{i=1}^n (x+\phi_i) = \sum_{j=0}^n\Phi_{n-j}^nx^j,
\]
that is,
\[
  \Phi_i^n \equiv \Phi_i^n(\phi)
  = \begin{cases}
  1, &i=0,
  \\
  \sum\limits_{1\le j_1< \cdots <j_i \le n} \phi_{j_1} \cdots \phi_{j_i},
  &1 \le i \le n.
  \end{cases}
\]

\begin{Lem}
\label{lrch}
With the above notation, we have
\begin{equation}
\label{dpbg}
  s_n^{\phi,u} = \sum_{j=0}^n \Phi_{n-j}^n s_j^{0,u}
\end{equation}
for each $n \ge 0$.
Moreover, for all $l \ge 1$ and $n\ge l-1$, we have
\begin{equation}
\label{rrgpl}
  s_{n+1}^{\phi,u}(x) = (1-x^2)(s_n^{\phi^{\{ l\}},u})'(x)
  + (\phi_l + u x) s_n^{\phi^{\{ l\}},u}(x),
\end{equation}
where $\phi^{\{l\}}$ is the sequence obtained by removing the $l$-th term $\phi_l$ from~$\phi$.
\end{Lem}

\begin{proof}
We have already obtained the identity
\[
  s_n^{\phi,u} = \sum_{j=0}^n \Phi_{n-j}^n \Lambda_u^j s_0.
\]
Now, a look at~\eqref{sn+1} shows that $\Lambda_u$ defines the polynomials $s_n^{0,u}$, that is,
$\Lambda_u^j s_0 = s_j^{0,u}$. This proves~\eqref{dpbg}. Since $\Lambda_u$ commutes with all the operators $\phi_j I$, the identity~\eqref{snLphi} can be reordered so as to get
\begin{align*}
  s_{n+1}^{\phi,u}
  &= (\Lambda_u + \phi_l I) (\Lambda_u + \phi_n I) \cdots
  (\Lambda_u + \phi_{l+1} I)(\Lambda_u + \phi_{l-1} I) \cdots (\Lambda_u + \phi_1 I) s_0
  \\
  &= (\Lambda_u + \phi_l I) s_n^{\phi^{\{ l\}},u},
\end{align*}
which is~\eqref{rrgpl}.
\end{proof}

For a positive integer $l$ and a real number $M$ write $\phi^{l,M}$ for the sequence
\begin{equation}
\label{sph}
  \phi^{l,M}_i = \phi_i+M\delta_{i,l}
\end{equation}
(where $\delta_{i,l}$ denotes de Kronecker delta).

\begin{Theo}
Let $\phi = (\phi_j)_{j\ge 1}$ be a sequence of real numbers and $u$ a real number such that $u + |\phi_j| < 0$ for all $j \ge 1$. Fix some $n \in \NN$. Then:
\begin{enumerate}
\item All the zeros of the polynomial $\spu_{n+1}$ are simple, belong to the interval $(-1,1)$ and, for each $l \ge 1$, interlace with the zeros of $s_n^{\phi^{\{l\}},u}$.

\item Let $\zeta_1(\phi) < \zeta_2(\phi) < \dots < \zeta_n(\phi)$ be the zeros of $\spu_n$. Then, $\zeta_k(\phi)$ is an increasing function of $\phi$. More precisely, if $\rho = (\rho_j)_{j\ge 1}$ is such that $\phi_j \le \rho_j$ and $u + |\rho_j| < 0$ for all $j \ge 1$, then $\zeta_k(\phi) \le \zeta_k(\rho)$.

\item Fix some $l \in \NN$ with $l \le n$ and some real number $M > 0$ such that $u + |\phi_l + M| < 0$. Let $\tilde{\phi} = (\tilde{\phi}_j)_{j \ge 1}$ be given by $\tilde{\phi}_l = \phi_l + M$ and $\tilde{\phi}_j = \phi_j$ for $j \ne l$. Then, the zeros of $\spu_n$ interlace the zeros of~$s_n^{\tilde{\phi},u}$.
\end{enumerate}
\end{Theo}

\begin{proof}
\begin{enumerate}
\item This follows from~\eqref{rrgpl} and Theorem~\ref{ptci}.

\item It is enough to prove the statement for sequences $\phi$ and $\rho$ differing only in one term, say $\rho_j = \phi_j + M$ for some $j \le n$ and some $M > 0$, and $\rho_k = \phi_k$ if $k \ne j$. The general case follows then by iteration.

If we abbreviate
$v_{n-1} = s_{n-1}^{\phi^{\{j\}},u}$, $s_n = \spu_n$, and $t_n = s_n^{\rho,u}$, 
identity~\eqref{rrgpl} gives
\begin{align*}
  s_n(x) &= (1-x^2) v_{n-1}'(x) + (\phi_j + ux) v_{n-1}(x),
  \\
  t_n(x) &= (1-x^2) v_{n-1}'(x) + (\phi_j + M + ux) v_{n-1}(x),
\end{align*}
that is,
\begin{equation}
\label{mcc}
  t_n(x) = s_n(x) + M v_{n-1}(x).
\end{equation}
Now, let $\xi_1 < \dots < \xi_{n-1}$, $\sigma_1 < \dots < \sigma_n$ and $\tau_1 < \dots < \tau_n$ be the zeros of $v_{n-1}$, $s_n$ and $t_n$, respectively. We want to prove that $\sigma_k < \tau_k$ for each $k$. By the interlacing property,
\begin{align*}
  -1 < \sigma_1 &< \xi_1, & \xi_k &< \sigma_{k+1} < \xi_{k+1},
  & \xi_{n-1} &< \sigma_n < 1 \\
  -1 < \tau_1 &< \xi_1, & \xi_k &< \tau_{k+1} < \xi_{k+1},
  & \xi_{n-1} &< \tau_n < 1.
\end{align*}
Observe also that under the condition $u + |\phi_n| < 0$, the recursive definition~\eqref{snLphi} of the polynomials $s_n^{\phi,u}$ gives $s_n(-1) > 0$, $t_n(-1) > 0$, and $v_{n-1}(-1) > 0$. Since all the zeros are simple, this leads to $\sgn v_{n-1} = (-1)^k$ in $(\xi_k,\xi_{k+1})$, and $\sgn s_n = (-1)^k$ in $(\sigma_k, \sigma_{k+1})$, so that, in $(\xi_k, \sigma_{k+1}]$,
\[
  t_n(x) = s_n(x) + M v_{n-1}(x) \ne 0,
\]
and this shows that $\sigma_{k+1} < \tau_{k+1}$. The proof that $\sigma_1 < \tau_1$ and $\sigma_n < \tau_n$ is similar.

\item This is actually what was proved above.
\qedhere
\end{enumerate}
\end{proof}

In particular, Theorem~\ref{ptci} gives the following.

\begin{Cor}
\label{ptc2}
Let $u<0$. Assuming that $|\phi_i| < -u$, $i\ge 1$, then all the zeros of polynomial $\spu_n$ are real, simple and live in $(-1,1)$. Moreover, the zeros of $s_{n+1}^{\phi,u}$ interlace the zeros of $\spu_n$.
\end{Cor}

This can be slightly improved.

\begin{Cor}
\label{ptc3}
Let $u<0$. Assume that $\spu_n (\pm 1) \ne 0$ and that the set $\{i:|\phi_i|> -u,\, 1\le i\le n\}$ has $N$ elements. Then the polynomial $\spu_n$ has at least $n-N$ zeros in $(-1,1)$.
\end{Cor}

\begin{proof}
We proceed by induction on $N$. The case $N=0$ is just Corollary~\ref{ptc2}.

Let us note that the hypothesis $\spu_n (\pm 1) \ne 0$ is equivalent to $|\phi_i| \ne -u$ (because of~\eqref{va-1} and~\eqref{va1}). And then $\spu_{n-1} (\pm 1) \ne 0$.

Assume next that the set $\{i:|\phi_i|> -u,\, 1\le i\le n\}$ has $N+1$ elements. We have to prove that $\spu_n$ has at least $n-N-1$ zeros in $(-1,1)$.

Since $\spu_n$ is a symmetric function of the parameters $\phi_i$, we can assume that they are ordered in increasing size of $|\phi_i|$, so that
$|\phi_i|< -u$, $1\le i\le n-N-1$ and $|\phi_i|> -u$, $n-N\le i\le n$.
The induction hypothesis gives that $\spu_{n-1}$ has at least $n-N-1$ zeros in $(-1,1)$.

Using then Lemma~\ref{depaq1}, we deduce that $\spu_{n}$ would have at least $n-N-2$ zeros in $(-1,1)$.
If $n-N-1$ is even, we deduce that $\spu_{n-1}(-1)$ and $\spu_{n-1}(1)$ have to have the same sign.
Since $|\phi_n|> -u$, we deduce that $\spu_{n}(-1)$ and $\spu_{n}(1)$ have to have the same sign (because of~\eqref{va-1} and~\eqref{va1}). Taking now into account that
$n-N-2$ is odd, we deduce that $\spu_n$ has to have at least $n-N-1$ zeros in $(-1,1)$. If $n-N-1$ is odd, we can proceed similarly.
\end{proof}

When the condition $|\phi_i| < -u$ fails in Corollary~\ref{ptc2}, the polynomials $\spu_n$ can have non real zeros.

The case when $\psi$ is constant and $\phi=(\phi_i)$ is constant for $i$ big enough is particularly interesting and will be studied in Section~\ref{sec6}. Here we prove the following equivalency.

\begin{Lem}
\label{eqp}
Let $u,v, K$ be real numbers with $K\in \NN$.
 Given sequences $\phi$, $\psi$, the following conditions are equivalent.
\begin{enumerate}
\item $\psi_i=u$, $i\ge 1$, and $\phi_i=v$, $i\ge K+1$.
\item There exist real numbers $\gamma_j$, $j=0,\dots, K$, with $\gamma_0=1$, $\gamma_K \ne 0$, such that
\[
  s_n^{\phi,\psi}(x) =\sum_{j=0}^K\gamma_j\scvu_{n-j}(x),\quad n\ge K,
\]
and the polynomial $P(x)=\sum_{j=0}^K\gamma_jx^{K-j}$ has only real zeros.
\end{enumerate}
\end{Lem}

\begin{proof}
First of all, we notice that if we have two constant sequences $\phi_i=v$, $\psi_i=u$, $1\le i$, and consider the polynomials
$\scvu_n(x)$, $n\ge0$, defined by~\eqref{sn+1}, then
using Lemma~\ref{lrch} for $\phi=v$, we get
\begin{equation}
\label{cdir}
  s_n^{v,u}(x) = \sum_{j=0}^n \binom{n}{j} v^js_{n-j}^{0,u}(x).
\end{equation}

(1) $\Rightarrow$ (2).
Consider the polynomial
\[
  P(x) = \prod_{i=1}^K(x+\phi_i-v) = \sum_{j=0}^K \gamma_jx^{K-j},
\]
which obviously has only real zeros.
Then for $n\ge K$,
\[
  \prod_{i=1}^n (x+\phi_i) = (x+v)^{n-K}P(x+v).
\]
Hence, using Lemma~\ref{lrch} and~\eqref{cdir} we have after easy computations
\[
  s_n^{\phi,\psi}(x) = \sum_{j=0}^K\gamma_j\scvu_{n-j}(x).
\]

(2) $\Rightarrow$ (1).
Let $\theta_i$ be the real zeros of $P$, $1\le i\le K$. Define then
\begin{equation*}
  \phi_i = \begin{cases} -\theta_i+v,&1\le i\le K,\\ v,&i\ge K+1.\end{cases}
\end{equation*}
The result again follows by using Lemma~\ref{lrch} and~\eqref{cdir}.
\end{proof}

\section{The case when \texorpdfstring{$\phi$}{phi} and \texorpdfstring{$\psi$}{psi} are constant sequences}
\label{sec5}

When the sequences $\phi$ and $\psi$ are constant, the polynomials $\spp_n$ enjoy a number of common properties, which include asymptotic properties for their zeros.

In this section is more convenient to perform the linear change of variable $x\mapsto 2x+1$, and
consider the polynomials $s_n^{v,u}(2x+1)$, $n\ge 0$. When $|v|+u<0$, as a consequence of Theorem~\ref{ptci}, we deduce that the polynomial $s_n^{v,u}(2x+1)$ has then $n$ real and simple zeros and they live in the interval $(-1,0)$. The reason to perform this change of variable is that the coefficients of $\scvu_n(2x+1)$ have the following appealing asymptotic behaviour (which will imply nice asymptotics for the zeros of $\scvu_n(2x+1)$).

\begin{Lem}
\label{l5.8}
Write
\begin{equation}
\label{coex}
  s_n^{v,u}(2x+1) = \sum_{j=0}^na_{n,j}^{v,u}x^j.
\end{equation}
If $u<0$ and $v+u<0$ then $(-1)^na_{n,j}>0$ and
\begin{equation}
\label{cotsv}
  1-j\left(\frac{-u-v+2j-2}{-u-v+2j}\right)^n
  \le \frac{(-1)^na_{n,j}^{v,u}}{\binom{-u+j-1}{j}(-u-v+2j)^n} \le 1
\end{equation}
for all $n\ge 0$ and $j=0,\dots, n$. Hence
\[
  \lim_{n\to \infty} \frac{a_{n,j}^{v,u}}{\binom{-u+j-1}{j}(u+v-2j)^n} = 1
\]
for all $j\ge 0$.
\end{Lem}

\begin{proof}
It is easy to see that the polynomials $f_n^{v,u}(x) = s_n^{v,u}(2x+1)$ satisfy the recurrence
\[
  f_{n+1}^{v,u}(x) = -2x(1+x)(f_n^{v,u})'(x)+(u+v+2ux)f_n^{v,u},
  \quad f_0^{v,u}=1.
\]
Hence~\eqref{gkp1} gives
\[
  a_{n+1,j}^{v,u} = (u+v-2j)a_{n,j}^{v,u}+2(u-j+1)a_{n,j-1}^{v,u},
  \quad 1\le j\le n+1.
\]
This easily shows that $(-1)^na_{n,j}\ge 0$.

Using \cite[Identity (2)]{Spi} (see also \cite{Neu}), we have for the $a_{n,j}^u$ the explicit expression
\begin{equation*}
  a_{n,j}^{v,u} = (-1)^n\binom{-u+j-1}{j} \sum_{l=0}^n 2^l(-u-v)^{n-l}\binom{n}{l}j!\,S(l,j),
\end{equation*}
where $S(l,j)$ denote the Stirling numbers of the second kind. Taking into account that for $l,j\ge 1$ (see \cite[p.~303]{Mez})
\[
  j^l\left(1-j\left(\frac{j-1}{j}\right)^l\right) \le j!\,S(l,j)\le j^l,
\]
from where the inequalities~\eqref{cotsv} follow easily.
\end{proof}

The main result of this section is Theorem~\ref{thchi}, which provides an asymptotic for the leftmost and rightmost zeros of $s_n^{v,u}(2x+1)$.
The theorem is a direct consequence of the following general result, which we will prove in the Appendix.

\begin{Theo}
\label{sip}
Let $F : \NN \cup \{0\} \to \RR \setminus \{0\}$ be a function such that
\begin{equation}
\label{ht41}
  |F(j-1) F(j+1)| \le F(j)^2
\end{equation}
for each $j \in \NN$. Let $G : \NN \cup \{0\} \to (0,+\infty)$ be a (positive) function such that
\[
  G(j-1) G(j+1) < G(j)^2
\]
for each $j \in \NN$.  Let $\tau_j(n) \in \RR$, for each $j, n \in \NN \cup \{0\}$, with these properties:
\begin{itemize}
\item there exists some constant $C > 0$ such that for each $j \le n \in \NN \cup \{0\}$,
\begin{equation}
\label{tauacotado}
  |\tau_j(n)| \le C;
\end{equation}
\item for each $j$,
\[
  \lim_{n \to +\infty} \tau_j(n) = 1.
\]
\end{itemize}
Let $K \in \RR$ (not necessarily positive) and, for each $n \in \NN$, let
\[
  P_n(x) = \sum_{j=0}^n F(j) G(j)^{n-K} \tau_j(n) x^j.
\]
Let $\zeta_{n,m}$, with $1 \le m \le n$, be the $m$-th real zero of $P_n$, ordered in increasing size of their absolute value. Then, for each $m \in \NN$,
\[
  \lim_{n\to+\infty}
  \frac{\zeta_{n,m}}{
  -\frac{F(m-1)}{F(m)} \left(\frac{G(m-1)}{G(m)}\right)^{n-K}
  }
  = 1.
\]
\end{Theo}

\begin{proof}[Proof of Theorem~\ref{thchi}]
In order to prove~\eqref{rmzi}, consider
\[
  F(j) = \binom{-u+j-1}{j},\quad G(j)=u+v-2j.
\]
It is just a matter of computation to check that $F$ and $G$ satisfy the hypothesis of Theorem~\ref{sip}. In particular, $-F(j-1)/F(j)=j/(u-j+1)$.

Write $\tau_n(j)=a_{n,j}^{v,u}/F(j)G^n(j)$, so that
\[
  s_n^{v,u}(x)=\sum_{j=0}^nF(j)G^n(j)\tau_n(j)x^j.
\]
Lemma~\ref{l5.8} shows that $\lim_n\tau_j(n)=1$, and $0<\tau_j(n)\le 1$.

Hence,~\eqref{rmzi} follows from Theorem~\ref{sip}.

In order to prove~\eqref{lmzi}, it is enough to consider the following symmetry
\begin{equation}
\label{cosxb}
  f_n^{v,u}(x) = (-1)^nf_n^{-v,u}(-1-x),
\end{equation}
where $f_n^{v,u}(x) = s_n^{v,u}(2x+1)$.
\end{proof}

For the sake of completeness, we display the generating function for the sequence $(s_n^{u,v})_n$,
\[
  \sum_{n=0}^\infty \frac{\scvu_n(x)}{n!} z^n 
  = (\cosh(z)+x\sinh(z))^u e^{vz} =: G_{v,u}(x,z),
\]
which is a straightforward consequence of the following partial differential equation for~$G_{v,u}$:
\[
  \frac{\partial}{\partial z} G_{v,u}(x,z)
  = (1-x^2)\frac{\partial}{\partial x} G_{v,u}(x,z) + (ux+v) G_{v,u}(x,z).
\]

\section{Zeros of linear combinations of the polynomials \texorpdfstring{$\scvu_n$}{snvu}}
\label{sec6}

Let $v$ and $u$ be real numbers satisfying $|v|+u<0$. If we fix a nonnegative integer $K$ and real numbers $\gamma_j$, $j=0,\dots, K$, with $\gamma_0=1$, $\gamma_K \ne 0$, in this section we consider the problem of determining the number of the real zeros of the polynomials
\begin{equation}
\label{polqn}
  q_n^{v,u}(x) = \sum_{j=0}^K \gamma_j \scvu_{n-j}(2x+1),\quad n\ge K.
\end{equation}
We will also prove some asymptotic for them.

As shown in Lemma~\ref{eqp} (except for the linear change of variables $x\to 2x+1$) these polynomials correspond with polynomials $(\spp_n)_n$ for the case when the sequence $\psi$ is constant and $\phi=(\phi_j)$ is constant for $j$ big enough, and the polynomial $\sum_{j=0}^K \gamma_j x^{K-j}$ has only real zeros.

Consider the first order differential operators defined by
\begin{align}
\notag
  \Lambda_{v,u} p(x)& = (1-x^2)p'(x) + (v+ux)p(x), \\
\label{tguv}
  \tilde \Lambda_{v,u} p(x)& = -2x(1+x)p'(x) + (u+v+2ux)p(x).
\end{align}
Writing $\tilde p(x)=p(2x+1)$, a simple computation shows that
\begin{equation}
\label{cuy1}
  \tilde \Lambda_{v,u}\tilde p=\widetilde{\Lambda_{v,u} p}.
\end{equation}

\begin{Lem}
\label{hqj2}
Let $v$ and $u$ be real numbers satisfying $|v|+u<0$.
\begin{enumerate}
\item\label{hqj2-1}
The operator $\tilde\Lambda_{v,u}$ is an increasing real zero operator (in the sense that the number of real zeros of the polynomial $\tilde \Lambda_{v,u}p$ is larger than the number of real zeros of~$p$).
\item\label{hqj2-2}
For $n\ge K$,
\begin{equation}
\label{cuy}
  q_n^{v,u} = \tilde \Lambda_{v,u}q_{n-1}^{v,u} = \tilde \Lambda_{v,u}^{n-K}q_K^{v,u}.
\end{equation}
\item\label{hqj2-3}
For $n\ge K$ the polynomial $q_n^{v,u}$ has at least $n-K$ real zeros.
\item\label{hqj2-4}
If there exists $n_0\ge K$ such that $q_{n_0}^{v,u}$ has only real and simple zeros, then $q_n^{v,u}$ has only real and simple zeros for $n\ge n_0$.
\end{enumerate}
\end{Lem}

\begin{proof}
Part~\eqref{hqj2-1} is consequence of~\eqref{cuy1} and Lemma~\ref{depaq2}.
Part~\eqref{hqj2-2} it follows easily using~\eqref{cuy1} and $\Lambda_{v,u}\scvu_n(x) = \scvu_{n+1}(x)$.
Part~\eqref{hqj2-3} is an easy consequence of part~\eqref{hqj2-1}. Part~\eqref{hqj2-4} follows easily using Lemma~\ref{depaq2}.
\end{proof}

As shown in Theorem~\ref{th6.3i}, we will prove much more on the zeros of $q_n^{v,u}$. To do that, we associate with the real numbers $\gamma_j$, $0\le j\le K$, the polynomial
\begin{equation}
\label{polp}
  P(x) = \sum_{j=0}^K\gamma_jx^{K-j}.
\end{equation}
Assume that $P$ has only real zeros and let $\theta_i$, $ 1\le i\le K$, be the zeros of the polynomial $P$. Define then (as in Lemma~\ref{eqp})
\begin{equation}
\label{phes}
  \phi_i = \begin{cases} -\theta_i,&1\le i\le K,\\ 0,&i\ge K+1.\end{cases}
\end{equation}
We will sometimes write $q_n^{\phi,v,u}=q_n^{v,u}$, in order to stress the dependence of $q_n^{v,u}$ on the sequence $\phi$ defined in~\eqref{phes}.

\begin{Rem}
\label{rv0}
As shown in Lemma~\ref{eqp},
\begin{equation}
\label{phex}
  q_n^{\phi,0,u}(x) = \spu_n(2x+1)
\end{equation}
and
\begin{equation}
\label{phexz}
q_n^{\phi, v,u}(x) = q_n^{\phi+v,0,u}(x) = s_n^{\phi+v,u}(2x+1).
  \end{equation}
And hence $q_n^{\phi, v,u}$ is the case $v=0$ associated to the polynomial $P(x+v)$.
As we see later on, taking this into account in order to prove Theorem~\ref{th6.3i}, it will be enough to consider the case $v=0$.
\end{Rem}

The polynomials $q_n^{\phi,v,u}$ enjoy the symmetry
\begin{equation}
\label{cosxc}
  q_n^{\phi,v,u}(x)=(-1)^{n}q_n^{-\phi,-v,u}(-1-x).
\end{equation}

Indeed, using~\eqref{cosxb}, we have
\begin{align*}
  q_n^{\phi,v,u}(x)&=\sum_{j=0}^K\gamma_j\scvu_{n-j}(2x+1) \\
  &= \sum_{j=0}^K (-1)^{n-j} \gamma_j s^{-v,u}_{n-j}(-2x-1)
  = (-1)^{n} q_n^{-\phi,-v,u}(-1-x).
\end{align*}

Part~\eqref{th6.3i-1} of Theorem~\ref{th6.3i} is now an easy consequence of~\eqref{phexz} and Theorem~\ref{ptci} (taking into account the change of variable $x\mapsto 2x+1$).

When $P$ has zeros outside of $(v+u,v-u)$, part~\eqref{th6.3i-1} of Theorem~\ref{th6.3i} does not apply and the polynomials $q_n^{v,u}$ might have non real zeros. Part~\eqref{th6.3i-2} of Theorem~\ref{th6.3i}, however, states that
for $n$ big enough (depending on the polynomial~$P$) all the zeros of $q_n^{v,u}$ are still real and completely describes where they live. For the rest of this section we address this.

We need the assumption 
\begin{equation}
\label{conc}
  P(u+v- 2l),P(-u-v+2l) \ne 0,\quad l\in \NN.
\end{equation}
Write $N^+$, $N^-$ for the cardinal of the sets
\begin{align}
\label{h+}
  \Hh ^+ &= \{l: \text{$l\ge 0$, $P(u+v-2l)P(u+v-2l-2)<0$}\}, \\
  \Hh^- &= \{l: \text{$l\ge 0$, $P(-u-v+2l)P(-u-v+2l+2)<0$}\},
\notag
\end{align}
respectively. Define finally
\begin{align*}
  \Gg^+ &= \{0,1,2,3,\dots \}\setminus \Hh^+, \\
  \Gg^- &= \{0,1,2,3,\dots \}\setminus \Hh^-.
\end{align*}

We start by proving the asymptotic behaviour of the zeros of $q_n^{0,u}$.
Definition~\eqref{phes} gives that the set $\{i: |\phi_i|>-u\}$ has at most $K$ elements. Hence~\eqref{phex} and Corollary~\ref{ptc3} then say that the number $n_-$ of negative zeros of $q_n^{0,u}$ in the interval $(-1,0)$ is at least $n-K$ (actually we will proof that for $n$ big enough $n_- = n-N^- - N^+$ and the zeros are simple). Write $\zeta^{\C}_{n,j}$, $1\le j\le n_-$, for the $n_-$ zeros of $q_n^{0,u}$ in $(-1,0)$ arranged in decreasing order (and taking into account their multiplicity).

\begin{Theo}
\label{onl3}
Let $\gamma_i$, $1\le i\le K$, be real numbers with $\gamma_0=1$, $\gamma_K \ne 0$. Assume that the polynomial $P$ \eqref{polp} satisfies~\eqref{conc}. If $u<0$, then, there exists $n_0$ such that for $n\ge n_0$ the polynomial $q_n^{0,u}$ \eqref{polqn} has exactly $N^+$ positive zeros and they are simple. Write then $\zeta^{\R}_{n,j}$, $1\le j\le N^+$, for the positive zeros of the polynomial $q_n^{0,u}$
arranged in increasing order.
Then for any given positive integer $j$, we have
\begin{align*}
  \lim_{n\to \infty}\frac{\zeta^{\R}_{n,j}}{\frac{h_j+1}{u-h_j}\frac{P(u-2h_j)}{P(u-2h_j-2)}
  \left(\frac{u-2h_j}{u-2h_j-2}\right)^{n-K}} &= 1, \quad 1\le j\le N^+,
  \\
  \lim_{n\to \infty} \frac{\zeta^{\C}_{n,j}}{\frac{g_j+1}{u-g_j} \frac{P(u-2g_j)}{P(u-2g_j-2)}
  \left(\frac{u-2g_j}{u-2g_j-2}\right)^{n-K}} &= 1, \quad 1\le j,
\end{align*}
where $\Hh^+=\{h_1,\dots, h_{N^+}\}$ and $\Gg^+ = \{g_1,g_2,\dots\}$ (arranged in increasing order).
Moreover, for any positive integer $j$, $n_0$ can also be taken so that for $n\ge n_0$ the $j$ rightmost negative zeros of the polynomial $q_n^{0,u}$ are simple.
\end{Theo}

Let us remark that in this theorem we do not need to assume that the polynomial $P$ has only real zeros.

\begin{proof}
Write $q_n^{0,u}(x) = \sum_{j=0}^nd_{n,j}x^j$. From~\eqref{polqn} and~\eqref{coex}, it follows
\[
  d_{n,j} = \sum_{l=0}^K\gamma_la_{n-l,j}.
\]
On the one hand, Lemma~\ref{l5.8} gives, for each~$j$,
\begin{equation}
\label{cotxa}
  d_{n,j} \sim \sum_{l=0}^K \gamma_l \binom{-u+j-1}{j} (u-2j)^{n-l}
  = \binom{-u+j-1}{j} (u-2j)^{n-K}P(u-2j).
\end{equation}

On the other hand, using Lemma~\ref{l5.8}, we get
\begin{align*}
  \left|\frac{d_{n,j}}{\binom{-u+j-1}{j}(u-2j)^{n-K}P(u-2j)}\right|
  & \le \frac{1}{|P(u-2j)|}
  \sum_{l=0}^K \frac{|\gamma_l(u-2j)^{K-l}|a_{n-l,j}}{\binom{-u+j-1}{j}(u-2j)^{n-l}} \\
  & \le \frac{\sum_{l=0}^K|\gamma_l(u-2j)^{K-l}|}{|\sum_{l=0}^K\gamma_l(u-2j)^{K-l}|}.
\end{align*}
The assumption~\eqref{conc} then gives
\begin{equation}
\label{cotxb}
  \left|\frac{d_{n,j}}{\binom{-u+j-1}{j}(u-2j)^{n-K} P(u-2j)}\right| \le C,
\end{equation}
for some constant $C$ depending on $P$ but not on $n$ or~$j$.
In order to prove both limits in Theorem~\ref{onl3}, consider
\[
  F(j) = \binom{-u+j-1}{j}P(u-2j),\quad G(j)=u-2j.
\]
We already know from the proof of Theorem~\ref{thchi} that $G$ satisfies the hypothesis of Theorem~\ref{sip}.
We next prove that for $j$ big enough $F$ satisfies
\[
  |F(j-1)F(j+1)| \le |F^2(j)|.
\]
Since $\binom{-u+j-1}{j}$ satisfies that hypothesis for all $j$, it will be enough to prove that for $j$ big enough
\[
  |P(u-2(j-1))P(u-2(j+1))| \le |P^2(u-2j)|.
\]
For the factors of $P$ corresponding to a real zero $\lambda$, we have
\[
  \frac{(u-2(j-1)-\lambda)(u-2(j+1)-\lambda)}{(u-2j-\lambda)^2}
  = \left(1-\frac{4}{(u-2j-\lambda)^2}\right) \le 1,
\]
for $j$ big enough. And for the factors of $P$ corresponding to a non-real zero $\lambda+i\mu$, we have
\begin{align*}
  & \frac{\bigl((u-2(j-1)-\lambda)^2+\mu^2\bigr) \bigl((u-2(j+1)-\lambda)^2+\mu^2\bigr)}
    {\bigl((u-2j-\lambda)^2+\mu^2\bigr)^2} \\
  & \qquad = \frac{\bigl((u-2j-\lambda)^2+\mu^2+4\bigr)^2 - 16(u-2j-\lambda)^2}{\bigl((u-2j-\lambda)^2+\mu^2]^2} \\
  & \qquad = 1+\frac{16+24\mu^2}{\bigl((u-2j-\lambda)^2+\mu^2\bigr)^2} - \frac{8}{(u-2j-\lambda)^2+\mu^2} 
  \le 1,
\end{align*}
for $j$ big enough.

Hence, for $j$ big enough (depending on $P$), $F$ satisfies the hypothesis~\eqref{ht41} of Theorem~\ref{sip}.
In particular,
\[
  -\frac{F(j-1)}{F(j)} = \frac{jP(u-2(j-1))}{(u-j+1)P(u-2j)}.
\]

Write $\tau_j(n)=d_{n,j}/F(j)G^{n-K}(j)$, so that
\[
  q_n^{0,u}(x) = \sum_{j=0}^nF(j)G^{n-K}(j)\tau_j(n)x^j.
\]
The estimate~\eqref{cotxa} shows that $\lim_n\tau_j(n)=1$, and~\eqref{cotxb} shows that $|\tau_j(n)|\le C$.

Hence, both limits in Theorem~\ref{onl3} follow from Remark~\ref{RemFjgrande} of Theorem~\ref{sip} (see the Appendix below).
\end{proof}

Using the symmetry~\eqref{cosxc}, we can also find the asymptotic for the leftmost zeros of $q_n^{0,u}$. In particular, we can prove that for $n$ big enough, $q_n^{0,u}$ has exactly $N^-$ real zeros in $(-\infty,-1)$ and they are simple.

We then have the following.

\begin{Cor}
\label{umv}
In the hypothesis of Theorem~\ref{onl3}, there exists $n_0$ such that if $n\ge n_0$ the polynomial $q_n^{0,u}$ has exactly $N^+$ positive (and simple) zeros and $N^-$ (simple) zeros in $(-\infty,-1)$, and if for some $n_1\ge n_0$, the zeros of
$q_{n_1}^{0,u}$ are real and simple, then for all $n\ge n_1$ the zeros of $q_n^{0,u}$ are also real and simple. Moreover, in that case the zeros of $q_{n+1}^{v,u}$ in $(-1,0)$, and $(0,+\infty)$ interlace the zeros of $q_n^{v,u}$ in $(-1,0)$ and $(0,+\infty)$, respectively, and the zeros of $q_{n}^{v,u}$ in $(-\infty,-1)$ interlace the zeros of $q_{n+1}^{v,u}$ in $(-\infty,-1)$.
\end{Cor}

\begin{proof}
We only have to prove the interlacing properties. But they follow by taking into account that (see~\eqref{cuy})
\[
  q_{n+1}^{0,u}(x) = -2x(1+x)(q_n^{0,u})'(x)+(u+2ux)q_n^{0,u}(x),
\]
and hence, if $\zeta$ is a zero of $q_n^{0,u}$ then
$q_{n+1}^{0,u}(\zeta)=-2\zeta(1+\zeta)(q_n^{0,u})'(\zeta)$.
\end{proof}

We are now ready to prove the part~\eqref{th6.3i-2} of Theorem~\ref{th6.3i}.

\begin{proof}[Proof of part~\eqref{th6.3i-2} of Theorem~\ref{th6.3i}]
As a consequence of Remark~\ref{rv0}, it is enough to prove the theorem for the case $v=0$. And, in light of Corollary~\ref{umv}, we only need to prove that, for $n$ big enough, all zeros of $q_n^{0,u}$ are real and simple.

We proceed by induction on $K$.
For $K=1$, the result is a direct consequence of Obreshkov's theorem (see \cite[p.~9]{Hou}) applied to $\scvu_n$, $\scvu_{n-1}$. Assume now that the theorem is true for~$K$.

Let $\phi_i$ be, $1\le i\le K+1$, real numbers with $\phi_i \ne \pm(u-2l)$, $l=0,1,\dots$, for $1\le i\le K+1$. If $\phi_i\in (u,-u)$, $1\le i\le K+1$, then the theorem follows from the part~\eqref{th6.3i-1} of Theorem~\ref{th6.3i}. Hence we assume that $\phi_{K+1}>-u$ (the case $\phi_{K+1}<u$ would then follow by using the symmetry~\eqref{cosxc}).

In order to simplify the notation we write $q_n=s_n^{\phi,0,u}$ and $p_n=s_n^{\phi^{\{K+1\}},0,u}$ (see the definition of $\phi^{\{l\}}$ in Lemma~\ref{lrch}).
Since $s_n^{\phi^{K+1,-\phi_{K+1}},0,u} = s_n^{\phi^{\{K+1\}},0,u} = p_n$ (see~\eqref{sph}), the following identity follows from~\eqref{mcc}:
\begin{equation}
\label{pmi}
  p_n(z) = q_n(x)-\phi_{K+1}p_{n-1}(x).
\end{equation}
The induction hypothesis says that the zeros of $p_n$ are simple and real for $n$ big enough. Using Theorem~\ref{onl3}, we write $\zeta_i^{\Ll}$, $1\le i \le N^-$, for the zeros of $p_n$ in $(-\infty,-1)$ arranged in decreasing order, $\zeta_i^{\C}$, $1\le i \le n-N^- - N^+$, for the zeros of $p_n$ in $(-1,0)$ arranged in decreasing order, and finally $\zeta_i^{\R}$, $1\le i\le N^+$, for the zeros of $p_n$ in $(0,+\infty)$ arranged in increasing order.

Setting $P(x) = \prod_{i=1}^K(x+\phi_i)$, we also have that $N^+$ is the number of elements of the set $\Hh^+$~\eqref{h+}.

The identity~\eqref{pmi} then gives
\begin{equation}
\label{esm}
  q_n(\zeta) = \phi_{K+1}p_{n-1}(\zeta)
\end{equation}
for any zero of~$p_n$.
Taking into account that the sign of the leading coefficient of $p_n$ is $(-1)^n$,
and that in $(-\infty,-1)$ the zeros of $p_{n-1}$ interlace the zeros of $p_n$, and in
$(-1,0)$ and $(0,+\infty)$ the zeros of $p_n$ interlace the zeros of $p_{n-1}$, respectively,
we get
\begin{align}
\label{esm1}
  \sgn\bigl(p_{n-1}(\zeta_{i}^{\Ll})\bigr) &= (-1)^{N^-+i+1},\quad 1\le i\le N^-, \\
\label{esm2}
  \sgn\bigl(p_{n-1}(\zeta_{i}^{\C})\bigr) &= (-1)^{n+N^++i},\quad 1\le i\le n-N^--N^+, \\
\label{esm3}
  \sgn\bigl(p_{n-1}(\zeta_{i}^{\R})\bigr) &= (-1)^{n+N^++i},\quad 1\le i\le N^+.
\end{align}
Hence,~\eqref{esm} and~\eqref{esm1} say that
\begin{equation}
\label{ton1}
\begin{aligned}
  q_n \text{ has an odd }&\text{number of zeros in} \\ 
  &\text{each interval } (\zeta_{i+1}^{\Ll},\zeta_i^{\Ll}), \ 1\le i\le N^--1,
\end{aligned}
\end{equation}
and so $q_n$ has at least $N^--1$ zeros in $(-\infty,-1)$. Similarly,
\eqref{esm} and~\eqref{esm2} say that
\begin{equation}
\label{ton2}
\begin{aligned}
  q_n \text{ has an odd } & \text{number of zeros in}\\ 
  &\text{each interval } (\zeta_{i+1}^{\C},\zeta_i^{\C}),\ 1\le i\le n-1-N^--N^+,
\end{aligned}
\end{equation}
that is, $q_n$ has at least $n-1-N^--N^+$ zeros in $(-1,0)$. And finally,~\eqref{esm} and~\eqref{esm3} say that
\begin{equation*}
\begin{aligned}
  q_n \text{ has an odd }&\text{number of zeros in} \\ 
  &\text{each interval } (\zeta_{i}^{\R},\zeta_{i+1}^{\R}),\ 1\le i\le N^+-1,
\end{aligned}
\end{equation*}
and so $q_n$ has at least $N^+-1$ zeros in $(0,+\infty)$.

Since $\phi_{K+1}>-u$, write $m_0=\lfloor (u+\phi_{K+1})/2 \rfloor$, so that $u-2m_0+\phi_{K+1} > 0 > u-2m_0-2+\phi_{K+1}$ ($\lfloor x\rfloor$ denotes the floor function).
Let $\tilde P(x)=\prod_{i=1}^{K+1}(x+\phi_i)$,
\begin{align*}
  \tilde \Hh^+ &= \{l: l \ge 0, \ \tilde P(u-2l)\tilde P(u-2l-2)<0\},\\
  \tilde \Hh^- &= \{l: l \ge 0, \ \tilde P(-u+2l)\tilde P(-u+2l+2)<0\}
\end{align*}
and denote by $\tilde N^+$ and $\tilde N^-$ the number of elements of $\tilde \Hh^+$ and $\tilde \Hh^-$, respectively. We have that $\tilde P(x)=(x+\phi_{K+1})P(x)$, and for $n$ big enough $\tilde N^+$ is the number of positive zeros of $q_n$, and $\tilde N^-$ is the number of zeros of $q_n$ in $(-\infty,-1)$ (see Theorem~\ref{onl3}).

We have to consider the change in the number of elements of $\tilde \Hh^+$ and $\Hh^+$, and similarly for $\tilde \Hh^-$ and~$\Hh^-$.
But since, for $0\le l$,
\[
  \sgn \bigl(\tilde P(-u+2l)\tilde P(-u+2l+2)\bigr) = \sgn\bigl(P(-u+2l) P(-u+2l+2)\bigr),
\]
we conclude that $\tilde \Hh^-=\Hh^-$. So we have only to discuss the following two cases:
\begin{enumerate}
\item $P(u-2m_0)P(u-2m_0-2) > 0$.
\item $P(u-2m_0)P(u-2m_0-2) < 0$.
\end{enumerate}

We start considering the first case: $P(u-2m_0)P(u-2m_0-2)>0$.
Then $\tilde P(u-2m_0)\tilde P(u-2m_0-2)<0$ and hence $\tilde N^+=N^++1$ and $\tilde N^-=N^-$. Theorem~\ref{onl3} says then that for $n$ big enough $q_n$ has to have $N^++1$ positive and simple zeros and $N^-$ zeros in $(-\infty,-1)$.
Moreover,~\eqref{ton2} gives that $q_n$ has at least $n-1-N^--N^+$ negative zeros in $(-\zeta_{n-N^--N^+}^{\C},\zeta_1^{\C})\subset (-1,0)$. This shows that $q_n$ has $n$ zeros and they have to be simple.

For the second case, we have $P(u-2m_0)P(u-2m_0-2)<0$, and so $\tilde P(u-2m_0)\tilde P(u-2m_0-2)>0$. Thus, $m_0 \in \Hh^+$ and
\[
  \tilde \Hh^+=\Hh^+\setminus\{m_0\}.
\]
As a consequence, we deduce that $\tilde N^+=N^+-1$. Write
\[
  \Gg^+ = \{1,2,3,\dots\}\setminus \Hh^+=\{g_i:1\le i\},
  \quad 
  \tilde \Gg^+ = \{1,2,3,\dots\}\setminus \tilde \Hh^+=\{\tilde g_i:1\le i\},
\]
with $g_i<g_j$, $\tilde g_i<\tilde g_j$ if $i<j$. It follows that $\tilde g_{m_0} < g_{m_0}$, because $\tilde \Gg^+=\Gg^+\cup \{m_0\}$.

Since $\tilde N^+=N^+-1$ and $\tilde N^-=N^-$, we have that $q_n$ has $N^+-1$ positive zeros and $N^-$ zeros in $(-\infty,-1)$ (and they are simple).

Write $\tilde \zeta_j^-$ for the negative zeros of $q_n$ arranged in decreasing order,
where either $1\le j\le n-1-N^+$ or $1\le j\le n+1-N^+$ (this is because
\eqref{ton2} gives that $q_n$ has at least $n-1-N^--N^+$ negative zeros in $(\zeta_{n-N^--N^+}^{\C},\zeta_1^{\C})\subset (-1,0)$ and $N^-$ zeros in $(-\infty,-1)$). This shows that $q_n$ has to have at least $n-2$ real zeros.

Theorem~\ref{onl3} says that for $n$ big enough, $\tilde \zeta_j^-$, $j=1,\dots, m_0$, are simple and belong the interval $(v_{n},0)$, where
\[
  v_{n} = (1+\eps)\frac{\tilde g_{m_0}+1}{u-\tilde g_{m_0}} 
  \frac{\tilde P(u-2\tilde g_{m_0})}{\tilde P(u-2\tilde g_{m_0}-2)}
  \left(\frac{u-2\tilde g_{m_0}}{u-2\tilde g_{m_0}-2}\right)^{n-K}.
\]
Also, Theorem~\ref{onl3} applied to $p_n$ gives that $\zeta_{m_0}^-<u_{n}$, where
\[
  u_{n} = (1-\eps) \frac{g_{m_0}+1}{u- g_{m_0}}
  \frac{P(u-2g_{m_0})}{P(u-2g_{m_0}-2)} \left(\frac{u-2 g_{m_0}}{u-2 g_{m_0}-2}\right)^{n-K}.
\]
We next assume the following \emph{claim}: $u_n < v_n$.

Hence, on the one hand, $q_n$ has $N^+-1$ simple and positive zeros. This gives that $q_n$ has to have at most $n-N^++1$ negative zeros.
$q_n$ has exactly $m_0$ negative zeros in $(v_n,0)$ and they are simple.
The claim implies that $q_n$ has at least $m_0$ simple and negative zeros in $(\zeta_{m_0}^-,0)$, and, since the number of negative zeros is at most $n-N^++1$, we deduce from~\eqref{ton1} and~\eqref{ton2} that $q_n$ has exactly $n-m_0-N^+$ negative zeros in $(-\infty,-1)\cup (\zeta_{n-N^--N^+}^{\C},\zeta_{m_0}^{\C})$, and they are simple. Hence $q_n$ has at least $n-N^+$ negative and simple zeros. Hence, we can conclude that $q_n$ has to have one more negative and simple zero.

We finish proving the claim. We have to prove that
\begin{multline*}
  (1-\eps) \frac{g_{m_0}+1}{u- g_{m_0}} \frac{P(u-2g_{m_0})}{P(u-2g_{m_0}-2)}
  \left(\frac{u-2 g_{m_0}}{u- 2g_{m_0}-2}\right)^{n-K}
  \\
  < (1+\eps)\frac{\tilde g_{m_0}+1}{u-\tilde g_{m_0}} 
  \frac{\tilde P(u-2\tilde g_{m_0})}{\tilde P(u-2\tilde g_{m_0}-2)}
  \left(\frac{u-2 \tilde{g}_{m_0}}{u-2\tilde{g}_{m_0}-2}\right)^{n-K}.
\end{multline*}
On the one hand, since $g_{m_0}\in \Gg^+$, we get $\frac{P(u-2g_{m_0})}{P(u-2
g_{m_0}-2)}>0$. And, on the other hand, since $\tilde g_{m_0}\in \tilde \Gg^+$ we have $\frac{\tilde P(u-2\tilde g_{m_0})}{\tilde P(u-2\tilde g_{m_0}-2)}>0$. 
Also, $u - g_{m_0} < 0$ and $u - \tilde{g}_{m_0} < 0$. Hence, the previous inequality follows for $n$ big enough (depending on $m_0$) because $\tilde g_{m_0} < g_{m_0}$, so that
\[
  \frac{u-2 g_{m_0}}{u- 2g_{m_0}-2}
  > \frac{u-2 \tilde{g}_{m_0}}{u-2\tilde{g}_{m_0}-2}.
\qedhere
\]
\end{proof}

The following example shows that if the assumption~\eqref{conc} fails, then the polynomials
$q_n^{u,v}$ can have non-real zeros for all $n$. The proof is similar to \cite[Lemma 10]{Duran-Bell}, but using the asymptotic of Lemma~\ref{l5.8}, so we omit it.

\begin{Lem}
Let $P(x) = (x-u-v+2l)^2$, $l\ge 1$, so that $\gamma_0=1$, $\gamma_1=-2(u+v-2l)$ and $\gamma_2=(u+v-2l)^2$, and consider the polynomials
\[
  q_n^{v,u}(x) = \sum_{j=0}^2 \gamma_j\scvu_{n-j}(2x+1).
\]
If $|v|+u<0$ then $q_n^{v,u}$ has always exactly two non-real zeros for all $n\ge 2$.
\end{Lem}

We guess that if $|v|+u<0$ and the polynomial $P$ \eqref{polp} satisfies~\eqref{conc}, the polynomial $q_n^{u,v}$ \eqref{polqn} will have only real zeros for $n$ big enough, even if $P$ has non-real zeros.

\begin{Conj}
If $|v|+u<0$ and $P(\pm(u+v-2l)) \ne 0$ for $l\in \NN$, then for $n$ big enough all the zeros of $q_n^{u,v}$ \eqref{polqn} are real.
\end{Conj}

The conjecture is saying that the first order differential operator
$\tilde \Lambda_{v,u}$ \eqref{tguv} is a real zero increasing operator in the following deeper sense than that established in the first part of Lemma~\ref{hqj2}. Indeed, take a monic polynomial $p$ of degree $K$, and write
\[
  p(x) = c\sum_{j=0}^K \gamma_j \scvu_{K-j}(x)
\]
(with $c$ chosen so that $\gamma_0=1$). It is easy to check that $q_n^{v,u} = \tilde \Lambda_{v,u}^n(p)$ (see~\eqref{cuy}).
Hence, if the polynomial $P(x) = \sum_{j=0}^K \gamma_{j}x^{K-j}$ satisfies $P(\pm(u+v-2l)) \ne 0$ for $l\in \NN$, the conjecture says that for $n$ big enough the polynomial $\tilde \Lambda_{v,u}^n(p)$ has only real zeros, no matter the number of real zeros of~$p$.

\section{Appendix}

The purpose of this section is to prove Theorem~\ref{sip}, which generalizes \cite[Theorem~4]{dzb}. This is a powerful result which provides asymptotic for the zeros
of polynomials $P_n(x) = \sum_{j=0}^n a_{n,j}x^j$
whose coefficients $a_{n,j}$ behave as $F(j)G^n(j)$, where $F$ and $G$ satisfy
$G(0),G(1)>0$ and
\[
  |F(j-1)F(j+1)|\le F^2(j),\quad 0 < G(j-1)G(j+1) < G^2(j).
\]
That is, $|F(j)|$ and $G(j)$ are logarithmically concave sequences.

We first prove that if we keep the degree fixed by taking $P_{N,n}(x)=\sum_{j=0}^Na_{n,j}x^j$ then the polynomials $P_{N,n}$ have simple and real zeros for $n$ big enough, and the zeros of $P_{N,n+1}$ interlace the zeros of $P_{N,n}$ (see Lemma~\ref{lemaQn>d}). The asymptotic properties for the zeros of $P_n$ then follow from the localization of the zeros of $P_{N,n}$. Actually what Theorem~\ref{sip} is saying is that the leftmost real zero $\zeta_{n,1}$ of $P_n$ behaves as the zero of the two first monomials of $P_n$, that is, as the zero of $a_{n,0}+a_{n,1}x$, so that
\[
  \zeta_{n,1}\sim -\frac{a_{n,0}}{a_{n,1}}\sim -\frac{F(0)}{F(1)}\left(\frac{G(0)}{G(1)}\right)^n.
\]
Similarly, the $m$-th leftmost real zero $\zeta_{n,m}$ of $P_n$ behaves as the zero of the two $m$-th monomials of $P_n$, that is, as the zero of $a_{n,m-1}+a_{n,m}x$, so that
\[
  \zeta_{n,m}\sim -\frac{a_{n,m-1}}{a_{n,m}}
  \sim -\frac{F(m-1)}{F(m)} \left(\frac{G(m-1)}{G(m)}\right)^n.
\]
The generality of Theorem~\ref{sip} far exceeds its application to the study of zeros of GKP polynomials. Indeed, in this paper we have used Theorem~\ref{sip} only for the very particular case of $G$ being a polynomial of degree~$1$.

We will need some previous lemmas in order to prove Theorem~\ref{sip}.

\begin{Lem}
\label{lemaG}
Let $G : \NN \cup \{0\} \to (0,+\infty)$ be a (positive) function such that
\[
  G(j-1) G(j+1) < G(j)^2
\]
for every $j \in \NN$. Then:
\begin{enumerate}
\item\label{Gj}
For each $m \in \NN$ and $j \in \NN \cup \{0\}$,
\[
  G(j) \left(\frac{G(m-1)}{G(m)}\right)^j
  \le \frac{G(m-1)^m}{G(m)^{m-1}}.
\]
Moreover, the left-hand side is strictly increasing in $j$ for $j \le m-1$ and strictly decreasing for $j \ge m$. The inequality is an equality for $j = m-1$ and $j = m$; otherwise, it is a strict inequality.

\item\label{limite0}
For each $m \in \NN$,
\[
  \lim_{j \to +\infty} G(j) \left(\frac{G(m-1)}{G(m)}\right)^j = 0.
\]
\end{enumerate}
\end{Lem}

Part~\eqref{Gj} is saying that for fixed $m$, the sequence $G(j) \left(\frac{G(m-1)}{G(m)}\right)^j$, $j\ge 1$, is unimodal with a plateau of two points $j=m-1,m$ (see \cite[Section 7.1]{Com} for the definition of unimodal sequences).

\begin{proof}
We start proving part~\eqref{Gj}. It is plain that the inequality is an equality for $j=m-1$ and $j=m$. So, we just need to prove that
\[
  G(j) \left(\frac{G(m-1)}{G(m)}\right)^j
\]
is strictly increasing in $j$ for $j \le m-1$ and strictly decreasing for $j \ge m$.

The assumption on $G$ means that the quotient $\frac{G(j)}{G(j+1)}$ is strictly increasing. If $j < m-1$, then
\[
  \frac{G(j)}{G(j+1)} < \frac{G(m-1)}{G(m)},
\]
which immediately gives
\[
  G(j) \left( \frac{G(m-1)}{G(m)} \right)^j < G(j+1) \left( \frac{G(m-1)}{G(m)} \right)^{j+1},
\]
as we wanted. If $j > m$, then the argument is the same, with the inequalities reversed.

As for part~\eqref{limite0}, let us fix $m \in \NN$. It follows from part~\eqref{Gj} that there exists the limit
\[
  \lim_{j \to +\infty} G(j) \left(\frac{G(m-1)}{G(m)}\right)^j = \ell_m \in [0,+\infty).
\]
Moreover, the assumption $G(m-1) G(m+1) < G(m)^2$ gives trivially that $\ell_m \le \ell_{m+1}$. Let us suppose then that $\ell_{m_0} > 0$ for some~$m_0$. Just by the definition of $\ell_{m_0}$ we have
\[
  1 = \lim_{j \to +\infty}
  \frac{G(j+1) \left(\frac{G(m_0-1)}{G(m_0)}\right)^{j+1}}
  {G(j) \left(\frac{G(m_0-1)}{G(m_0)}\right)^j},
\]
that is,
\[
  \lim_{j \to +\infty} \frac{G(j+1)}{G(j)} = \frac{G(m_0)}{G(m_0-1)}.
\]
Since the limit on the left-hand side does not depend on $m_0$ and $\ell_{m_0} \le \ell_{m}$ for each $m > m_0$, it follows that
\[
  \frac{G(m)}{G(m-1)} = \frac{G(m_0)}{G(m_0-1)},
\]
that is,
\[
  G(m) = c \lambda^m
\]
for each $m \ge m_0$, with some constants $c, \lambda > 0$. But this contradicts the hypothesis $G(m-1) G(m+1) < G(m)^2$. Therefore, $\ell_m = 0$ for each $m \in \NN$.
\end{proof}

In the sequel, given two intervals $I, J \subseteq \RR$, we will write $I \prec J$ when $\sup I < \inf J$.

\begin{Lem}
\label{lemaImn}
Let $F : \NN \cup \{0\} \to \RR \setminus \{0\}$ be an arbitrary function.
Let $G : \NN \cup \{0\} \to (0,+\infty)$ be a (positive) function such that
\[
  G(j-1) G(j+1) < G(j)^2
\]
for every $j \in \NN$.  Fix $K \in \RR$. For each $m \in \NN$, $n \in \NN$ and $\eps \in (0,1)$, let
\begin{multline}
  I_{m,n}
  = \Bigg\{ x \in \RR : \Bigg| \frac{x}{-\left|\frac{F(m-1)}{F(m)}\right| 
  \left(\frac{G(m-1)}{G(m)}\right)^{n-K}} - 1 \Bigg| < \eps
  \Bigg\}
  \\
  = \biggl(
  -(1+\eps) \left|\tfrac{F(m-1)}{F(m)}\right| \left(\tfrac{G(m-1)}{G(m)}\right)^{n-K},
  -(1-\eps) \left|\tfrac{F(m-1)}{F(m)}\right| \left(\tfrac{G(m-1)}{G(m)}\right)^{n-K}
  \biggr)
\label{Imn}
\end{multline}
and
\begin{multline}
  I_{m,n}'
  = \Bigg\{ x \in \RR : \Bigg|
  \frac{x}{-\left|\tfrac{F(m)}{F(m+1)}\right| \tfrac{m}{m+1} \left(
  \frac{G(m)}{G(m+1)}
  \right)^{n-K}}
  -1
  \Bigg| < \eps
  \Bigg\}
  \\
  = \biggl(
  -(1+\eps) \left|\tfrac{F(m)}{F(m+1)}\right| \tfrac{m}{m+1} \left(
  \tfrac{G(m)}{G(m+1)}
  \right)^{n-K},
  -(1-\eps) \left|\tfrac{F(m)}{F(m+1)}\right| \tfrac{m}{m+1} \left(
  \tfrac{G(m)}{G(m+1)}
  \right)^{n-K}
  \biggr)
\label{Imn'}
\end{multline}
(for simplicity, we only stress the dependence on $m$ and~$n$).
Then:
\begin{enumerate}
\item
\label{lemaImn-1}
Given $m$, for every $\eps < \frac{1}{2m+1}$ and every $n \in \NN$, $I_{m+1,n} \prec I_{m,n}'$.
\item
\label{lemaImn-2}
Given $m$ and $\eps_0 < 1$, there exists some $n_0 \in \NN$ such that for every $n > n_0$ and $0 < \eps < \eps_0$, $I_{m,n}' \prec I_{m,n}$.
\end{enumerate}
\end{Lem}

\begin{proof}
The proof of part~\eqref{lemaImn-1} is as follows. Given $m, n \in \NN$ and $\eps > 0$,
\begin{multline*}
  I_{m+1,n} \prec I_{m,n}'
  \\
  \iff
  -(1-\eps) \left|\frac{F(m)}{F(m+1)}\right| \frac{G(m)^{n-K}}{G(m+1)^{n-K}}
  < -(1+\eps) \left|\frac{F(m)}{F(m+1)}\right| \frac{m}{m+1}
  \frac{G(m)^{n-K}}{G(m+1)^{n-K}}
  \\
  \iff
  (1+\eps) \frac{m}{m+1} < 1-\eps
  \iff \eps < \frac{1}{2m+1}.
\end{multline*}

We next prove part~\eqref{lemaImn-2}. Given $m, n \in \NN$ and $0 < \eps < 1$,
\begin{multline*}
  I_{m,n}' \prec I_{m,n}
  \\
  \iff
  -(1-\eps) \left|\frac{F(m)}{F(m+1)}\right| \frac{m}{m+1}
  \frac{G(m)^{n-K}}{G(m+1)^{n-K}}
  < -(1+\eps) \left|\frac{F(m-1)}{F(m)}\right| \frac{G(m-1)^{n-K}}{G(m)^{n-K}}
  \\
  \iff \frac{1+\eps}{1-\eps}
  \left|\frac{F(m-1) F(m+1)}{F(m)^2}\right|
  \frac{m+1}{m}
  < \left(
  \frac{G(m)^2}
  {G(m-1) (G(m+1))}
  \right)^{n-K}.
\end{multline*}
Since, given~$m$,
\[
  1 < \frac{G(m)^2}{G(m+1) G(m-1)}
\]
and
\[
  \frac{1+\eps}{1-\eps} < \frac{1+\eps_0}{1-\eps_0}
\]
if $0 < \eps < \eps_0$,
the above inequality holds for $n$ large enough and $0 < \eps < \eps_0$.
\end{proof}

According to the lemma, given $m$ and $\eps_0 \le \frac{1}{2m+1}$, the intervals $I_{k,n}$ and $I_{k,n}'$ are arranged as follows for $k=1,2,\dots m$, $0 < \eps < \eps_0$ and $n > n_0$:
\[
  I_{m,n} \prec \dots \prec I_{k+1,n}' \prec I_{k+1,n} \prec I_{k,n}' \prec I_{k,n} \prec \dots \prec I_{3,n} \prec I_{2,n}' \prec I_{2,n} \prec I_{1,n}' \prec I_{1,n}
\]

For the next lemma, let us consider the intervals
\begin{equation}
\label{Jmn}
  J_{m,n} = \Bigg\{ x \in \RR : \Bigg| \frac{x}{-\frac{F(m-1)}{F(m)} 
   \left(\frac{G(m-1)}{G(m)}\right)^{n-K}} - 1 \Bigg| < \eps
  \Bigg\}.
\end{equation}
Each $J_{m,n}$ is the interval with endpoints
\[
  -(1 \pm \eps) \frac{F(m-1)}{F(m)} \left(\frac{G(m-1)}{G(m)}\right)^{n-K}.
\]
Given some $m \in \NN$, if $\frac{F(m-1)}{F(m)} > 0$ we have $J_{m,n} = I_{m,n}$, while $J_{m,n} = -I_{m,n}$ if $\frac{F(m-1)}{F(m)} < 0$. The above lemma proves that, in any case, given $m$ the intervals $J_{1,n}, \dots, J_{m,n}$ are disjoint for $\eps$ small enough and $n$ large enough.

The next lemma is the key to proving Theorem~\ref{sip}. The case $G(j)=j$ and $F(j)$ being a multiplier sequence is \cite[Lemma 6]{dzb}.

\begin{Lem}
\label{lemaQn>d}
Let
\begin{itemize}
\item $F : \NN \cup \{0\} \to \RR \setminus \{0\}$ be an arbitrary function,

\item $G : \NN \cup \{0\} \to (0,+\infty)$ be a (positive) function such that for every $j \in \NN$,
\[
  G(j-1) G(j+1) < G(j)^2,
\]

\item $K$ be a real number (not necessarily positive),

\item $\tau_j(n) \in \RR$, for each $j, n \in \NN \cup \{0\}$, $j \le n$, such that for each $j$,
\[
  \lim_{n \to +\infty} \tau_j(n) = 1.
\]
\end{itemize}
For each $n, N \in \NN \cup \{0\}$, let
\[
  Q_{n,N}(x) = \sum_{j=0}^N F(j) G(j)^{n-K} \tau_j(n) x^j.
\]
Given $\eps \in (0,1)$ and $N \in \NN$, there exist some $n_0$ and $\delta > 0$ such that for every $n \ge n_0$:
\begin{enumerate}
\item
The polynomial $Q_{n,N}$ has exactly one (simple) zero in each interval $J_{m,n}$ ($m=1,2,\dots,N$) defined by~\eqref{Jmn}.

\item
$|Q_{n,N}(x)| \ge \delta G(0)^{n-K}$, if $x \notin \bigcup_{m=1}^N J_{m,n}$.
\end{enumerate}
\end{Lem}

\begin{proof}
Let us fix $\eps \in (0,1)$ and $N \in \NN$; for now, take any~$n$. Let us evaluate $Q_{n,N}$ in both endpoints of an interval $J_{m,n}$, with $m=1,2,\dots,N$:
\begin{multline*}
  Q_{n,N}\left( -(1\pm \eps) \tfrac{F(m-1)}{F(m)} \left( \tfrac{G(m-1)}{G(m)} \right)^{n-K} \right)
  \\
  = \sum_{j=0}^N (-1)^j (1\pm \eps)^j F(j) \tfrac{F(m-1)^j}{F(m)^j}
  \tau_j(n) G(j)^{n-K} \left( \tfrac{G(m-1)}{G(m)} \right)^{j(n-K)}
\end{multline*}
(take apart the terms $j=m-1$ and $j=m$)
\begin{multline*}
  = (-1)^{m-1} (1\pm \eps)^{m-1} \tfrac{F(m-1)^m}{F(m)^{m-1}}
  \tau_{m-1}(n) \left( \tfrac{G(m-1)^m}{G(m)^{m-1}} \right)^{n-K}
  \\
  + (-1)^m (1\pm \eps)^m \tfrac{F(m-1)^m}{F(m)^{m-1}}
  \tau_m(n) \left( \tfrac{G(m-1)^m}{G(m)^{m-1}} \right)^{n-K}
  \\
  + \sum_{\substack{0 \le j \le N \\ j\ne m-1, m}} (-1)^j (1\pm \eps)^j F(j) \tfrac{F(m-1)^j}{F(m)^j}
  \tau_j(n) G(j)^{n-K} \left( \tfrac{G(m-1)}{G(m)} \right)^{j(n-K)}
\\
  = (-1)^{m-1} (1\pm \eps)^{m-1} \tfrac{F(m-1)^m}{F(m)^{m-1}}
  \left( \tfrac{G(m-1)^m}{G(m)^{m-1}} \right)^{n-K}
  \big( \tau_{m-1}(n) - (1 \pm \eps) \tau_m(n) \big)
  \\
  + \sum_{\substack{0 \le j \le N \\ j\ne m-1, m}} (-1)^j (1\pm \eps)^j F(j) \tfrac{F(m-1)^j}{F(m)^j}
  \tau_j(n) \left( G(j) \left(\tfrac{G(m-1)}{G(m)}\right)^j \right)^{n-K}.
\end{multline*}
Dividing by $\left( \tfrac{G(m-1)^m}{G(m)^{m-1}} \right)^{n-K}$ we obtain that
\begin{multline*}
  \left( \tfrac{G(m)^{m-1}}{G(m-1)^m} \right)^{n-K} 
  Q_{n,N}\left( -(1\pm \eps) \tfrac{F(m-1)}{F(m)} \left( \tfrac{G(m-1)}{G(m)} \right)^{n-K} \right)
  \\
  = (-1)^{m-1} (1\pm \eps)^{m-1} \tfrac{F(m-1)^m}{F(m)^{m-1}}
  \big( \tau_{m-1}(n) - (1 \pm \eps) \tau_m(n) \big)
  \\
  + \sum_{\substack{0 \le j \le N \\ j\ne m-1, m}} (-1)^j 
  (1\pm \eps)^j F(j) \tfrac{F(m-1)^j}{F(m)^j}
  \tau_j(n) \left(
  \frac{G(j) \Big(\tfrac{G(m-1)}{G(m)}\Big)^j}
  {\tfrac{G(m-1)^m}{G(m)^{m-1}}}
  \right)^{n-K}.
\end{multline*}
Let us take now limit in $n$, taking into account that:
\begin{itemize}
\item For each $j$, $\lim\limits_{n\to +\infty} \tau_j(n) = 1$.

\item According to Lemma~\ref{lemaG}, part~\eqref{Gj}, $\frac{G(j) \left(\tfrac{G(m-1)}{G(m)}\right)^j}
  {\tfrac{G(m-1)^m}{G(m)^{m-1}}} < 1$ for $j \ne m-1, m$.
\end{itemize}
Therefore, for each $m \in \{1,2,\dots,N\}$,
\begin{multline}
\label{limQnM}
  \lim_{n\to +\infty} \left( \tfrac{G(m)^{m-1}}{G(m-1)^m} \right)^{n-K} 
  Q_{n,N}\left( -(1\pm \eps) \tfrac{F(m-1)}{F(m)} 
  \left( \tfrac{G(m-1)}{G(m)} \right)^{n-K} \right)
  \\
  = (-1)^{m-1} (1\pm \eps)^{m-1} \tfrac{F(m-1)^m}{F(m)^{m-1}}
  (\mp \eps).
\end{multline}
This implies that, given $m \in \{1,2,\dots,N\}$ and taking $n$ large enough, both values
\begin{equation}
\label{QImndistintosigno}
  Q_{n,N}\left( -(1\pm \eps) \tfrac{F(m-1)}{F(m)} \left( \tfrac{G(m-1)}{G(m)} \right)^{n-K} \right)
\end{equation}
have opposite sign, that is, $Q_{n,N}$ changes its sign at the endpoints of $J_{m,n}$, so it has at least one zero in this interval. For $n$ large enough, this happens in all the $J_{m,n}$, with $m=1,2,\dots,N$. Since $Q_{n,N}$ has degree $N$, these zeros are unique (and simple). This proves the first part of the lemma.

If all zeros of $Q_{n,N}$ are real and simple, then all zeros of its derivative $Q_{n,N}'$ are simple as well, and there is exactly one zero between any two consecutive zeros of $Q_{n,N}$. Between any two consecutive zeros, $|Q_{n,N}|$ increases up to some point, then it decreases. In particular, in any interval $[a,b]$ on which $Q_{n,N}$ does not vanishes, it is plain that
\[
  |Q_{n,N}(x)| \ge \min\{ |Q_{n,N}(a)|, |Q_{n,N}(b)|\},
  \quad x \in [a,b].
\]

Thus, if $n$ is so large that the values~\eqref{QImndistintosigno} have opposite sign for $m=1,2,\dots,N$, it follows that
\[
  |Q_{n,N}(x)| \ge \min_{m=1,2,\dots,N}\left\{
  \left|Q_{n,N}\left( -(1\pm \eps) \tfrac{F(m-1)}{F(m)} 
  \left( \tfrac{G(m-1)}{G(m)} \right)^{n-K} \right)\right|
  \right\}
\]
for $x \notin \bigcup_{m=1}^N J_{m,n}$. According to Lemma~\ref{lemaG}, part~\eqref{Gj},
we have that
\[
  1 \ge G(0) \frac{G(m)^{m-1}}{G(m-1)^m}.
\]
Therefore, keeping in mind the limit~\eqref{limQnM}, we conclude that if $n$ is large enough, then for every $x \notin \bigcup_{m=1}^N J_{m,n}$,
\begin{multline*}
  |Q_{n,N}(x)| \ge \min_{m=1,2,\dots,N}\Big\{
  \Big|Q_{n,N}\Big( -(1\pm \eps) \tfrac{F(m-1)}{F(m)} \Big( \tfrac{G(m-1)}{G(m)} \Big)^{n-K} \Big)\Big|
  \Big\}
  \\
  \ge G(0)^{n-K} \min_{m=1,2,\dots,N}\Big\{
  \Big( \tfrac{G(m)^{m-1}}{G(m-1)^m} \Big)^{n-K}
  \Big| Q_{n,N}\Big( -(1\pm \eps) \tfrac{F(m-1)}{F(m)} 
  \Big( \tfrac{G(m-1)}{G(m)} \Big)^{n-K} \Big)\Big|
  \Big\}
  \\
  \ge G(0)^{n-K} \min_{m=1,2,\dots,N}\Big\{
  \frac{1}{2} \Big|
  (-1)^{m-1} (1\pm \eps)^{m-1} \tfrac{F(m-1)^m}{F(m)^{m-1}}
  (\mp \eps)
  \Big|
  \Big\}
  \\
  = G(0)^{n-K} \min_{m=1,2,\dots,N} \Big\{ \frac{\eps}{2} (1 - \eps)^{m-1} 
  \Big|\tfrac{F(m-1)^m}{F(m)^{m-1}} \Big| \Big\}.
\end{multline*}
This proves the second part of the lemma, with
\[
  \delta = \min_{m=1,2,\dots,N} \left\{
  \frac{\eps}{2} (1 - \eps)^{m-1} \left|\tfrac{F(m-1)^m}{F(m)^{m-1}}\right|
  \right\}.
  \qedhere
\]
\end{proof}

\begin{Rem}
\label{taumenos}
Instead of the condition that, for every $j$,
\[
  \lim_{n \to +\infty} \tau_j(n) = 1,
\]
we can impose these slightly weaker conditions:
\begin{itemize}
\item For each $j$,
\[
  \lim_{n\to+\infty} \frac{\tau_j(n)}{\tau_{j+1}(n)} = 1.
\]
\item For each $j$, there exists some constant $C_j > 0$ such that, for every $n \in \NN$, $C_j \le \tau_j(n)$.
\end{itemize}
The proof needs only minor changes, so we omit it.
\end{Rem}

We are now ready to prove Theorem~\ref{sip} (the case $G(j)=j$ and $F(j)$ being a multiplier sequence is~\cite[Theorem~4]{dzb}).

\begin{proof}[Proof of Theorem~\ref{sip}]
Let us fix $m \in \NN$ and $\eps \in (0,1)$. According to Lemma~\ref{lemaG}, part~\eqref{limite0}, we can take $N \in \NN$ with the condition that $m \le N$ and, for reasons to be seen later, such that
\begin{equation}
\label{menorque12}
  (1+\eps) G(N+1)
  \left(\frac{G(m-1)}{G(m)}\right)^{N+1}
  < \frac{1}{2} G(0).
\end{equation}
Let us take $n_0$ and $\delta$ as given by Lemma~\ref{lemaQn>d}.
For each $n \ge N+1$, let us decompose
\[
  P_n(x) = Q_{n,N}(x) + R_{n,N}(x),
\]
with
\begin{align*}
  Q_{n,N}(x) &= \sum_{j=0}^N F(j) G(j)^{n-K} \tau_j(n) x^j,
  \\
  R_{n,N}(x) &= \sum_{j=N+1}^n F(j) G(j)^{n-K} \tau_j(n) x^j.
\end{align*}
We are going to prove first that if $n$ is large enough and
\begin{equation}
\label{rangodex}
  |x| \le (1+\eps) \left|\frac{F(m-1)}{F(m)}\right|
  \left( \frac{G(m-1)}{G(m)} \right)^{n-K},
\end{equation}
then
\[
  |R_{n,N}(x)| \le \frac{\delta}{2} G(0)^{n-K}.
\]
Let us consider each term of the polynomial $R_{n,N}$. If $N+1 \le j \le n$, then $j \ge m$ as well. And, by the property of the function~$F$,
\[
  \left|\frac{F(j)}{F(j-1)}\right|
  \le \left|\frac{F(j-1)}{F(j-2)}\right| 
  \le \dots \le \left|\frac{F(m)}{F(m-1)}\right|.
\]
Therefore,
\begin{align*}
  |F(j)|
  &= \left|\frac{F(j)}{F(j-1)} \cdot \frac{F(j-1)}{F(j-2)} \dots \frac{F(m)}{F(m-1)} F(m-1)\right|
  \\
  &\le \left| \frac{F(m)}{F(m-1)} \right|^{j-m+1} |F(m-1)|.
\end{align*}
Under the condition~\eqref{rangodex},
\begin{multline*}
  |F(j) G(j)^{n-K} \tau_j(n) x^j|
  \\
  \le \left| \frac{F(m)}{F(m-1)} \right|^{j-m+1} |F(m-1)| G(j)^{n-K} |\tau_j(n)|
  (1+\eps)^j
  \left|\frac{F(m-1)}{F(m)}\right|^j
  \left( \frac{G(m-1)}{G(m)} \right)^{(n-K)j}
  \\
  = \left| \frac{F(m-1)}{F(m)} \right|^{m-1} |F(m-1)| G(j)^{n-K} |\tau_j(n)|
  (1+\eps)^j
  \left( \frac{G(m-1)}{G(m)} \right)^{(n-K)j}
  \\
  \le \left| \frac{F(m-1)}{F(m)} \right|^{m-1} |F(m-1)| \, |\tau_j(n)|
  \left(
  (1+\eps)
  G(j)
  \left( \frac{G(m-1)}{G(m)} \right)^j
  \right)^{n-K} (1+\eps)^K.
\end{multline*}
According to Lemma~\ref{lemaG}, part~\eqref{Gj}, and since $j \ge N+1 > m$, we have that
\[
  G(j)
  \left( \frac{G(m-1)}{G(m)} \right)^j
  \le G(N+1) \left( \frac{G(m-1)}{G(m)} \right)^{N+1}.
\]
Keeping in mind~\eqref{tauacotado} and~\eqref{menorque12}, it follows that
\[
  |F(j) G(j)^{n-K} \tau_j(n) x^j|
  \le C \left| \frac{F(m-1)}{F(m)} \right|^{m-1} |F(m-1)| 
  \frac{1}{2^{n-K}} G(0)^{n-K} (1+\eps)^K
\]
and thus,
\[
  |R_{n,N}(x)| \le
  C \left| \frac{F(m-1)}{F(m)} \right|^{m-1} |F(m-1)| \frac{n}{2^{n-K}} G(0)^{n-K}
  (1+\eps)^K.
\]
As a consequence, if $n$ is large enough (depending on $m$ and $\eps$), it follows that
\[
  |R_{n,N}(x)| \le \frac{\delta}{2} G(0)^{n-K},
\]
as we wanted. According to Lemma~\ref{lemaImn}, we can assume that the intervals $I_{k,n}$ and $I_{k,n}'$ ($k=1,2,\dots,m$) defined by~\eqref{Imn} and~\eqref{Imn'} are arranged as follows:
\[
  I_{m,n} \prec \dots \prec I_{k+1,n}' \prec I_{k+1,n} \prec I_{k,n}' \prec I_{k,n} 
  \prec \dots \prec I_{3,n} \prec I_{2,n}' \prec I_{2,n} \prec I_{1,n}' \prec I_{1,n}
\]

According to Lemma~\ref{lemaQn>d}, in each interval $J_{k,n}$ given by~\eqref{Jmn} the polynomial $Q_{n,N}$ has just one simple zero. Moreover, $Q_{n,N}$ has opposite sign on the endpoints of each $J_{k,n}$.

In $(\inf I_{m,n},-\inf I_{m,n}) \setminus \bigcup_{k=1}^m J_{k,n}$, we have that $|Q_{n,N}| \ge \delta G(0)^{n-K}$, according to Lemma~\ref{lemaQn>d}, and $|R_{n,N}(x)| \le \frac{\delta}{2} G(0)^{n-K}$, as we have just proven. Therefore, $P_n$ has there the same sign as $Q_{n,N}$. In particular, it has no zeros on $(\inf I_{m,n},-\inf I_{m,n}) \setminus \bigcup_{k=1}^m J_{k,n}$. Moreover, $P_n$ has opposite sign on the endpoint of each $J_{k,n}$, so it has at least one zero in each $J_{k,n}$, $k=1,2,\dots,m$.

Let us note now that
\begin{align*}
  P_n'(x) &= \sum_{j=1}^n F(j) j G(j)^{n-K} \tau_j(n) x^{j-1}
  = \sum_{j=0}^{n-1} F(j+1)(j+1) G(j+1)^{n-K} \tau_{j+1}(n) x^j
  \\
  &= \sum_{j=0}^{n-1} \widetilde{F}(j) \widetilde{G}(j)^{n-1-K} \tau_{j+1}(n) x^j,
\end{align*}
where
\begin{align*}
  \widetilde{F}(j) &= F(j+1) (j+1) G(j+1),
  \\
  \widetilde{G}(j) &= G(j+1).
\end{align*}
The functions $\widetilde{F}$ and $\widetilde{G}$ trivially satisfy the same conditions posed on $F$ and $G$, respectively. Therefore, we can apply what is already proven to the polynomials $P_{n+1}'$, changing $F$ to $\widetilde{F}$ and $G$ to~$\widetilde{G}$: for $n$ large enough, the first negative zeros of~$P_n'$ (change also $n$ to $n-1$) belong to the intervals
\begin{multline*}
  \Bigg\{ x \in \RR : \Bigg| \frac{x}{-\frac{\widetilde{F}(k-1)}{\widetilde{F}(k)} 
  \Big(\frac{\widetilde{G}(k-1)}{\widetilde{G}(k)}\Big)^{n-1-K}} - 1 \Bigg| < \eps
  \Bigg\}
  \\
  = \Bigg\{ x \in \RR : \Bigg| \frac{x}{-\frac{F(k) k}{F(k+1) (k+1)} 
  \Big(\frac{G(k)}{G(k+1)}\Big)^{n-K}} - 1 \Bigg| < \eps
  \Bigg\},
\end{multline*}
which, as follows easily from Lemma~\ref{lemaImn}, are disjoint with the intervals~$J_{k,n}$.

Therefore, $P_n'$ does not vanish on any $J_{k,n}$ and $P_n$ has in each $J_{k,n}$ just one zero, which is simple.
This means that $\zeta_{n,m} \in J_{m,n}$ for $n$ large enough. In other words,
\[
  \Bigg|
  \frac{\zeta_{n,m}}{
  -\frac{F(m-1)}{F(m)} \Big(\frac{G(m-1)}{G(m)}\Big)^{n-K}
  } -1
  \Bigg|
  < \eps,
\]
as we wanted to show.
\end{proof}

\begin{Rem}
\label{RemFjgrande}
The condition
\[
  |F(j-1) F(j+1)| \le F(j)^2
\]
is required only for $j$ large enough, say $j \ge m_0$. Reviewing the proof, we need to take then $m \ge m_0$. But this is not a problem, since in fact the proof says that for each $k \le m$ the $k$-th zero of $P_n$ is simple and belongs to~$J_{k,n}$.
\end{Rem}

\begin{Rem}
The proof does not need the zeros of $P_n$ and $P_{n-1}$ to interlace. But, as a consequence of the theorem, it is easy to deduce that they interlace, if $F$ is a positive function (so that $I_{k,n} = J_{k,n}$): we just need to check that the intervals $I_{m,n}$ and $I_{m,n-1}$ interlace. It looks as if both properties were equivalent: that the zeros are simple and that the zeros of $P_n$ and $P_{n-1}$ interlace.
\end{Rem}

\begin{Rem}
The $\frac{1}{2}$ appearing in~\eqref{menorque12} can be changed by any $s \in (0,1)$. Repeating the arguments, one then arrives at
\[
  |R_{n,N}(x)| \le
  \left| \frac{F(m-1)}{F(m)} \right|^{m-1} |F(m-1)| n s^{n-K} G(0)^{n-K} (1+\eps)^K 
  \sup_{0 \le j \le n} |\tau_j(n)|.
\]
Thus, the condition~\eqref{tauacotado} is much more than needed. Instead of the assumption $|\tau_j(n)| \le C$ of Theorem~\ref{sip}, it would be enough $|\tau_j(n)| \le C^n$ for some constant $C > 0$, for instance. Neither this nor the weaker conditions of Remark~\ref{taumenos} would affect the conclusions of Theorem~\ref{sip}.
\end{Rem}

\paragraph{\textbf{Acknowledgments.}}

The research of the first author was partially supported by PID2024-155593NB-C21 (MICIU and Feder Funds (European Union)), IMUS-Mar\'ia de Maeztu grant CEX2024-001517-M (funded by MICIU/\allowbreak AEI/\allowbreak 10.13039/\allowbreak 501100011033), Pro\-gra\-ma Ope\-ra\-ti\-vo PPIT-FEDER Anda\-lu\-c\'ia 2021-2027 SOL2024-31596 and SOL2024-31708 and FQM-262 (Jun\-ta de Anda\-lu\-c\'ia). 
The research of the second author was partially supported by DGA Project E48\_23R.
The research of the second and third authors was partially supported by PID2024-155593NB-C22.



\end{document}